\documentclass[12pt]{article}
\usepackage{amsfonts}
\usepackage[T1]{fontenc}
\usepackage{amsfonts}
\usepackage{amssymb}
\usepackage{bbm}
\usepackage{amscd}
\usepackage{mathrsfs}
\usepackage{amsmath,amssymb, amsthm}
\usepackage{graphicx}
\usepackage{color}
\usepackage{authblk}
\usepackage{subcaption}
\usepackage{indentfirst,latexsym,amssymb}
\usepackage[bookmarks=false]{hyperref}
\usepackage{tikz}

\usepackage[normalem]{ulem}

\newtheorem{theorem}{Theorem}[section]
\newtheorem*{directp}{Theorem \ref{directp}}
\newtheorem*{dbundle}{Theorem \ref{dbundle}}
\newtheorem*{ncon}{Theorem \ref{ncon}}
\newtheorem*{hmmtype}{Theorem \ref{hmmtype}}
\newtheorem{pro}[theorem]{Proposition}
\newtheorem{conj}[theorem]{Conjecture}
\newtheorem{lem}[theorem]{Lemma}
\newtheorem{coro}[theorem]{Corollary}
\theoremstyle{definition}

\newtheorem{defi}[theorem]{Definition}

\newtheorem{ques}[theorem]{Question}

\newtheorem{exam}[theorem]{Example}
\newtheorem{remark}[theorem]{Remark}
\def\Sym{\mathrm{Sym}}
\def\Aut{\mathrm{Aut}}

\def\Cay{\mathrm{Cay}}

\makeatletter

\newcommand{\Rmnum}[1]{\expandafter\@slowromancap\romannumeral #1@}
\makeatother

\begin{document}
\title{Graph product and the stability of circulant graphs}
\date{}

\author{Junyang Zhang}
\affil{\small School of Mathematical Sciences~\& Chongqing Key Lab of Cognitive Intelligence and Intelligent Finance, Chongqing Normal University, Chongqing 401331, P. R. China}

\maketitle

\renewcommand{\thefootnote}{\fnsymbol{footnote}}
 \footnotetext{E-mail addresses: jyzhang@cqnu.edu.cn (Junyang Zhang)}
\begin{abstract}
A graph $\Gamma$ is said to be stable if
$\mathrm{Aut}(\Gamma\times K_2)\cong\mathrm{Aut}(\Gamma)\times \mathbb{Z}_{2}$ and unstable  otherwise. If an unstable graph is connected, non-bipartite and any two of its distinct vertices have different neighbourhoods, then it is called nontrivially unstable. We establish conditions guaranteeing the instability of various graph products, including direct products, direct product bundles, Cartesian products, strong products, semi-strong products, and lexicographic products. Inspired by a condition for the instability of direct product bundles, we propose a new sufficient condition for circulant graphs to be unstable and refine existing instability conditions from the literature. Based on these results, we categorize unstable circulant graphs into two distinct types and further propose a classification framework.


\end{abstract}

\section{Introduction}
Given a set $S$, a subset $R$ of $S\times S$ is called a \emph{binary relation} on $S$. We write $a\sim_R b$ (respectively, $a\nsim_R b$) to denote that $(a,b)\in R$ (respectively, $(a,b)\notin R$). For brevity, we refer to a binary relation simply as a relation in this paper. A relation $R$ on $S$ is said to be \emph{irreflexive} if $(a,a)\notin R$ for all $a\in S$. The dual relation $R^{\ast}$ of $R$ is defined as $R^{\ast}:=\{(a,b)\mid (b,a)\in R\}$. If $R^{\ast}=R$, then $R$ is called a \emph{symmetric} relation.

\medskip
Throughout this paper, all graphs are assumed to be finite, undirected, and simple. Thus, a graph $\Gamma$ is represented as an ordered pair $(V(\Gamma),E(\Gamma))$, where $V(\Gamma)$ is the vertex set and $E(\Gamma)$ is the edge set. Here, $E(\Gamma)$ is an irreflexive symmetric relation on $V(\Gamma)$. For convenience, we write $u\sim_{\Gamma} v$ instead of $u\sim_{E(\Gamma)} v$ to indicate that $(u,v)\in E(\Gamma)$. For each $v\in V(\Gamma)$, define
\begin{equation*}
\Gamma(v):=\{u\in V(\Gamma)\mid u\sim_{\Gamma} v\}~\mbox{and}~ \Gamma[v]:=\Gamma(v)\cup\{v\}.
\end{equation*}
The set $\Gamma(v)$ is called the \emph{neighbourhood} of $v$ in $\Gamma$. The \emph{automorphism group} of $\Gamma$, denoted by $\Aut(\Gamma)$, consists of all permutations of $V(\Gamma)$
that preserve adjacency of $\Gamma$. We use exponential notation for group actions: if $\alpha$ is a permutation of a set $S$ and $x\in S$, then $x^{\alpha}$ denotes the image of $x$ under $\alpha$. For undefined group theoretical and graph theoretical terminologies, we refer the reader to \cite{R1995} and \cite{BM2008}, respectively.

\medskip
The \emph{direct product} $\Gamma\times\Sigma$ of two graphs $\Gamma$ and $\Sigma$ is defined as a graph with vertex set $V(\Gamma)\times V(\Sigma)$ and $(a,x)\sim_{\Gamma\times\Sigma}(b,y)\Longleftrightarrow a\sim_{\Gamma}b~\mbox{and}~x\sim_{\Sigma}y$ for every pair of vertices
$(a,x),(b,y)\in V(\Gamma\times\Sigma)$. Note that an element $(\sigma,\rho)$ in the direct product $\Aut(\Gamma)\times\Aut(\Sigma)$ of the two groups $\Aut(\Gamma)$ and $\Aut(\Sigma)$ can be seen as a permutation on $V(\Gamma\times\Sigma)$ such that $(a,x)^{(\sigma,\rho)}=(a^{\sigma},x^{\rho})$ for all
$(a,x)\in V(\Gamma\times\Sigma)$. Since
\begin{align*}
  (a^{\sigma},x^{\rho})\sim_{\Gamma\times\Sigma}
  (b^{\sigma},y^{\rho})
  \Longleftrightarrow&
  a^{\sigma}\sim_{\Gamma}b^{\sigma}~\mbox{and}~
  x^{\rho}\sim_{\Sigma}y^{\rho} \\
  \Longleftrightarrow&
  a\sim_{\Gamma}b~\mbox{and}~x\sim_{\Sigma}y \\
  \Longleftrightarrow& (a,x)\sim_{\Gamma\times\Sigma} (b,y)
\end{align*}
for every pair of vertices $(a,x),(b,y)\in V(\Gamma\times\Sigma)$, we treat $\Aut(\Gamma)\times\Aut(\Sigma)$ as a subgroup of $\Aut(\Gamma\times \Sigma)$. The graph pair $(\Gamma,\Sigma)$ is said to be \emph{stable} if $\Aut(\Gamma\times \Sigma)=\Aut(\Gamma)\times\Aut(\Sigma)$ and \emph{unstable}  otherwise \cite[Definition 1.1.]{QXZZ2021}. In particular,  $\Gamma$ is said to be \emph{stable} (respectively, \emph{unstable}) if the graph pair $(\Gamma,K_2)$ is stable (respectively, unstable), where $K_2$ is the complete graph with two vertices. Recall that the concept of stability of a graph was first given by Maru\v si\v c et al. \cite{MSZ1989}.

\medskip
A graph is called \emph{$R$-thick} if there exists a pair of distinct vertices which have the same neighbourhood and \emph{$R$-thin} otherwise \cite{HIK2011}. An \emph{$R$-thin} graph is also said to be vertex-determining \cite{MSZ1989}, worthy \cite{W2008} or twin-free \cite{KL1996}. It is straightforward to check that a graph is unstable whenever it is disconnected, or $R$-thick, or a bipartite graph admitting a nontrivial automorphism. If an unstable graph is connected, $R$-thin and non-bipartite, then it is called \emph{nontrivially} unstable \cite{W2008}.

\medskip
The stability of graphs has received considerable attention in the literature. In \cite{MSZ1992}, Maru\v si\v c et al. extended the concept of Cayley graphs to that of generalized Cayley graphs and proved that every generalized Cayley graph which is not a Cayley graph is unstable. In \cite{NS1996}, it was shown by Nedela and \v Skoviera that the stability of graphs played an important role in finding regular embeddings of canonical double covers on orientable surfaces. In \cite{S2001}, Surowski
illustrated some methods for constructing arc-transitive nontrivially unstable graphs and constructed three infinite families of such graphs as applications.
In \cite{W2008}, Wilson demonstrated three sufficient conditions for a graph to be unstable, and constructed several infinite families of nontrivially unstable graphs.  Additional results on this topic can be found in  \cite{AKK2023,FH2022,HM2024,HMM2021,
HMM2023,HK2023,QXZ2019,QXZ2021,WM2021,WM2023}.

\medskip
Classifying all nontrivially unstable members within an infinite family of graphs is extremely difficult. Such a classification remains incomplete even for circulant graphs. In \cite[Theorems C.1--C.4]{W2008}, Wilson proposed four sufficient conditions for the instability of circulant graphs. Two of these conditions were found to be flawed and remedied by Qin et al. \cite{QXZ2019} and Hujdurovi\'c et al. \cite{HMM2021} respectively. Building on this, Hujdurovi\'c et al. \cite[Theorem 3.2, Proposition 3.7 and Proposition 3.12]{HMM2021} established novel sufficient conditions for the instability of circulant graphs. They proved that any circulant graph satisfying one of Wilson's conditions must also satisfy at least one of theirs, and demonstrated the existence of circulant graphs that satisfy their new conditions yet lie outside the scope of Wilson's framework.

\medskip
A key question in the study of circulant graphs is whether the conditions proposed by Hujdurovi\'c et al. characterize all nontrivially unstable cases. This problem can be approached in two ways: either by confirming the conditions for specific families of circulant graphs or by exploring new conditions that lead to the instability of circulant graphs.

\medskip
For the first approach, existing results show that no nontrivially unstable circulant graphs exist within arc-transitive families \cite{QXZ2019} or those of odd order
\cite{FH2022}. Additionally, all unstable circulant graphs of valency at most $7$ satisfy Wilson's conditions C.1 or C.4 \cite{HMM2023}. Recent work by Hujdurovi\'c and Kov\'acs \cite{HK2023} further demonstrates that nontrivially unstable circulant graphs of order twice a prime power must comply with the conditions in \cite[Theorem C.1]{W2008} or \cite[Proposition 3.7]{HMM2023}.

\medskip
For the second approach, Jiang and Zhang \cite{JZ2025} established a condition under which a Cayley graph is unstable. When applied to circulant graphs, this condition reduces to Proposition \ref{preserving}, serving as a generalization of one of a criteria of  Hujdurovi\'c et al (\cite[Proposition 3.12]{HMM2021}). However, we have not yet succeeded in using this condition to construct an unstable circulant graph that violates any of the conditions of Hujdurovi\'c et al.. Actually, through extensive computational efforts spanning several days,  Hujdurovi\'c et al. \cite[Remark 6.2]{HMM2021} identified all nontrivially unstable circulant graphs of order up to $50$ and confirmed that each of these graphs adheres to at least one of their proposed conditions.

\medskip
The lack of counterexamples with order $\leq 50$ or degree $\leq 7$ to the conditions proposed by Hujdurovi\'c et al. complicates computer-aided efforts to investigate the second approach (exploring new conditions).
To address this challenge, we focus on foundational graph constructions by investigating the stability of several kinds of graph products: direct products, direct product bundles, Cartesian products, strong products, semi-strong products, and lexicographic products. Through this framework, we develop multiple families of nontrivially unstable circulant graphs. Analysis of one such family reveals a new sufficient condition for instability in circulant graphs. Building on this, we refine the conditions proposed by Hujdurovi\'c et al. and formulate several conjectures regarding the stability of circulant graphs.

\medskip
Let $\Gamma$ and $\Sigma$ be two graphs. The \emph{Cartesian product} $\Gamma\Box\Sigma$, the \emph{strong product} $\Gamma\boxtimes\Sigma$, the \emph{semi-strong product} $\Gamma\ltimes\Sigma$ and the \emph{lexicographic product} $\Gamma[\Sigma]$ are all graphs with vertex set $V(\Gamma)\times V(\Gamma)$. Their adjacency relations are defined as follows: for any two vertices $(a,x),(b,y)\in V(\Gamma)\times V(\Gamma)$,
\begin{equation*}
(a,x)\sim_{\Gamma\Box\Sigma}(b,y)\Longleftrightarrow
  a=b~\mbox{and}~x\sim_{\Sigma}y~\mbox{or}~ a\sim_{\Gamma}b~\mbox{and}~x=y;
\end{equation*}
\begin{equation*}
(a,x)\sim_{\Gamma\boxtimes\Sigma}(b,y)\Longleftrightarrow a=b~\mbox{and}~x\sim_{\Sigma}y,~\mbox{or}~a\sim_{\Gamma}b ~\mbox{and}~x\sim_{\Sigma}y,~\mbox{or}~
  a\sim_{\Gamma}b~\mbox{and}~x=y;
\end{equation*}
\begin{equation*}
(a,x)\sim_{\Gamma\ltimes\Sigma}(b,y)\Longleftrightarrow
  a\sim_{\Gamma}b~\mbox{or}~a=b,~\mbox{and}~x\sim_{\Sigma}y;
\end{equation*}
\begin{equation*}
(a,x)\sim_{\Gamma[\Sigma]}(b,y)\Longleftrightarrow
  a\sim_{\Gamma}b,~\mbox{or}~a=b~\mbox{and}~x\sim_{\Sigma}y.
\end{equation*}

\medskip
A graph is \emph{prime} with respect to the direct product (respectively, Cartesian product) if it is nontrivial---meaning it has at least two vertices---and cannot be factored into a direct product (respectively, Cartesian product) of two nontrivial graphs.
Two graphs $\Gamma$ and $\Sigma$ are \emph{coprime}
with respect to the direct product (respectively, Cartesian product) if there exists no nontrivial graph
 $\Lambda$ such that $\Gamma=\Gamma_1\times \Lambda$ and $\Sigma=\Sigma_1\times \Lambda$ (respectively,  $\Gamma=\Gamma_1\Box \Lambda$ and $\Sigma=\Sigma_1\Box \Lambda$) for some graphs $\Gamma_1$ and $\Sigma_1$.

\medskip
Our results on direct product involve the concept of a Cartesian skeleton, which was first introduced non-algorithmically in \cite{HIK2009}. To exclude graphs with loops, we instead adopt the definition given in \cite{QXZ2019}.
As explained in \cite{QXZ2019}, this definition agrees with that in \cite{HIK2009}.
The \emph{Boolean square} $B(\Gamma)$ of a graph $\Gamma$ is the graph with vertex set $V(\Gamma)$ and
\begin{equation*}
u\sim_{B(\Gamma)}v\Longleftrightarrow u\neq v~\mbox{and}~\Gamma(u)\cap \Gamma(v)\ne \emptyset
\end{equation*}
for all $u,v\in V(\Gamma)$. An edge of $B(\Gamma)$ with ends $u$ and $v$ is said to be \emph{dispensable} with respect to $\Gamma$ if there exists $w\in V(\Gamma)$ such that
\begin{equation*}
  \Gamma(u)\cap \Gamma(v)\subsetneq \Gamma(u)\cap \Gamma(w)~\mbox{or}~
  \Gamma(u)\subsetneq\Gamma(w)\subsetneq \Gamma(v)
\end{equation*}
and
\begin{equation*}
  \Gamma(u)\cap \Gamma(v)\subsetneq \Gamma(v)\cap \Gamma(w)~\mbox{or}~
  \Gamma(v)\subsetneq\Gamma(w)\subsetneq \Gamma(u).
\end{equation*}
The \emph{Cartesian skeleton} $S(\Gamma)$ of $\Gamma$ is the spanning subgraph of $B(\Gamma)$ obtained by removing from $B(\Gamma)$ all dispensable edges with respect to $\Gamma$.

\medskip
We are now ready to state our main results of this paper. For direct product, we prove the following theorem.

\begin{theorem}
\label{directp}
Let $\Gamma$ and $\Sigma$ be two graphs. We have
\begin{enumerate}
  \item if $\Gamma$ or $\Sigma$ is unstable, then $\Gamma\times\Sigma$ is unstable;
  \item if $S(\Sigma)$ and $S(\Gamma)$ are coprime with respect to Cartesian product and $\Gamma$ and $\Sigma$ are both stable, then $\Gamma\times\Sigma$ is stable;
  \item if $S(\Gamma)$ and $S(\Sigma)$ are coprime with respect to Cartesian product and $\Gamma\times\Sigma$ is nontrivially unstable, then either $\Gamma$ or  $\Sigma$ is nontrivially unstable.
\end{enumerate}
\end{theorem}
Theorem \ref{directp} (iii) generalizes
\cite[Lemma 4.3]{QXZ2019}, while the proof of Theorem \ref{directp} (ii) relies on a result by Qin et al. concerning the Cartesian skeleton of a graph (see Lemma \ref{tfsk}).

\medskip
Our next two theorems require the concepts of two-fold morphism (TF-morphism) and two-fold semi-morphism (TFS-morphism). Recall that the concept of a two-fold morphism was first introduced in \cite{LMS2014}, where it was termed a two-fold automorphism.
\begin{defi}
Let $\Gamma$ be a graph and $\alpha$ and $\beta$ be two permutations on $V(\Gamma)$.
\begin{enumerate}
  \item If $u\sim_{\Gamma}v$ implies $u^\alpha\sim_{\Gamma}v^\beta$ for any two vertices $u,v\in V(\Gamma)$, then we call the ordered pair $(\alpha,\beta)$ a \emph{two-fold morphism} of $\Gamma$.
  \item  If $u^\alpha\sim_{\Gamma}u^\beta$ for every $u\in V(\Gamma)$ and  $u^\alpha\sim_{\Gamma}v^\beta$ or $u^\alpha=v^\beta$ for each pair of adjacent vertices $u$ and $v$ of $\Gamma$, then we call the ordered pair $(\alpha,\beta)$ a \emph{two-fold semi-morphism} of $\Gamma$.
\end{enumerate}
\end{defi}
\begin{remark}
\label{rtfs}
\begin{enumerate}
  \item Obviously, $(\alpha,\alpha)$ is a TF-morphism of $\Gamma$ if and only if $\alpha\in \Aut(\Gamma)$. A TF-morphism $(\alpha,\beta)$ of $\Gamma$ is said to be \emph{nontrivial} if $\alpha\neq \beta$.
  \item Let $E^{\circ}=\{(u,v)\mid u=v~\mbox{or}~u\sim_{\Gamma}v,~u,v\in V(\Gamma)\}$.
If $(\alpha,\beta)$ is a TFS-morphism of $\Gamma$, then $(\alpha,\beta)$ preserves $E^{\circ}$ and the restriction of $(\alpha,\beta)$ to $E^{\circ}$ is injective. Since $E^{\circ}$ is finite, we have that $(\alpha,\beta)$ acts as a permutation on $E^{\circ}$. It follows that
 \begin{equation*}
u\sim_{\Gamma}v\Longleftrightarrow u^\alpha=v^\beta~\mbox{or}~
u^\alpha\sim_{\Gamma}v^\beta
\end{equation*}
for any two distinct vertices $u,v\in V(\Gamma)$.

Similarly, if $(\alpha,\beta)$ is a TF-morphism of $\Gamma$, then $(\alpha,\beta)$ acts as a permutation on the set $\{(u,v)\mid u\sim_{\Gamma}v,~u,v\in V(\Gamma)\}$. Thus
$u\sim_{\Gamma}v\Longleftrightarrow u^\alpha\sim_{\Gamma}v^\beta$ for any two distinct vertices $u,v\in V(\Gamma)$.
\end{enumerate}
\end{remark}

\medskip
Our second theorem is dedicated to direct product  bundles of two graphs $\Gamma$ and $\Sigma$, which are a variant of direct product graphs. Let $p$ be a mapping from $V(\Gamma)\times V(\Gamma)$ to $\Aut(\Sigma)$ such that $p(a,b)=p(b,a)^{-1}$ for every $(a,b)\in V(\Gamma)\times V(\Gamma)$. Then the \emph{direct product bundle} $\Gamma\times^{p}\Sigma$ is the graph with vertex set $V(\Gamma)\times V(\Sigma)$ such that $(a,x)\sim_{\Gamma\times^{p}\Sigma}(b,y)\Longleftrightarrow
  a\sim_{\Gamma}b~\mbox{and}~x\sim_{\Sigma}y^{p(a,b)^{-1}}$ for all $(a,x),(b,y)\in V(\Gamma)\times V(\Sigma)$ \cite{KM1995}.

\begin{theorem}
\label{dbundle}
Let $\Gamma$ and $\Sigma$ be two graphs and $p$ be a mapping from $V(\Gamma)\times V(\Gamma)$ to $\Aut(\Sigma)$ such that $p(a,b)=p(b,a)^{-1}$. Then  $\Gamma\times^{p}\Sigma$ is unstable if one of the following two statements holds:
\begin{enumerate}
  \item $\Gamma$ has a nontrivial TF-morphism $(\alpha,\beta)$ such that
  $p(a^{\alpha},b^{\beta})=p(a,b)$ for every $(a,b)\in V(\Gamma)\times V(\Gamma)$;
  \item there is a TF-morphism $(\alpha,\beta)$ of $\Sigma$ and a permutation $\theta$ on $V(\Sigma)$ such that $\alpha\neq\theta$ and $\theta p(a,b)^{-1}=p(a,b)^{-1}\beta$ for every $(a,b)\in V(\Gamma)\times V(\Gamma)$.
\end{enumerate}
\end{theorem}
The complement $\overline{\Gamma}$ of a nontrivial graph $\Gamma$ is the graph with vertex set $V(\Gamma)$ such that $u\sim_{\overline{\Gamma}}v$ if and only if $u\neq v$ and $u\nsim_{\Gamma}v$ for all $u,v\in V(\Gamma)$. For other kinds of graph products, we have the following theorem.

\begin{theorem}
\label{others}
Let $\Gamma$ and $\Sigma$ be two connected nontrivial graphs. Then we have
\begin{enumerate}
\item if one of $\Gamma$ and $\Sigma$ is bipartite and the other has a nontrivial TF-morphism, then $\Gamma\Box\Sigma$ is unstable;
 \item if one of $\Gamma$ and $\Sigma$ is bipartite and the other is $R$-thick, then $\Gamma\boxtimes\Sigma$ is unstable;
  \item if $\Gamma$ or $\Sigma$ has a TFS-morphism, then $\overline{\Gamma\boxtimes\Sigma}$ is unstable;
  \item if $\Sigma$ is unstable or $\Gamma$ has a TFS-morphism, then $\Gamma\ltimes\Sigma$ is unstable;
  \item if $\Sigma$ has a nontrivial TF-morphism, then $\Gamma[\Sigma]$ is unstable.
\end{enumerate}
\end{theorem}

\medskip
We demonstrate that Theorems \ref{directp}, \ref{dbundle} and \ref{others} yield infinitely many nontrivially unstable circulant graphs. Through analyzing  a specific family of these graphs, we derive a sufficient condition for the instability of circulant graphs, which is presented as follows.
\begin{theorem}
\label{ncon}
Let $\Gamma=\Cay(\mathbb{Z}_{2m},S)$ be a circulant graph. If $\Cay\big(\mathbb{Z}_{2m},(S\setminus\{m\})+m\big)$ has an automorphism fixing $0$ but moving $m$, then $\Gamma$ is unstable.
\end{theorem}

We failed to construct an unstable circulant graph  violating the conditions of Hujdurovi\'c et al. via graph products. On the other hand, some existing instability conditions for circulant graphs are too cumbersome to be practical. Additionally, multiple overlapping conditions may apply to the same unstable circulant graph, further complicating their use. We prove (see Lemma \ref{equi}) that for a circulant graph $\Gamma=\Cay(\mathbb{Z}_{2m}, S)$ satisfying the condition in \cite[Theorem 3.2]{HMM2021}, the following holds: if $m\in H\setminus K_{o}$, then $\Gamma$ must satisfy the condition in \cite[Proposition 3.7]{HMM2021}; if $4\mid|H|$ and $m\notin H\setminus K_{o}$, then $\Gamma$ must satisfy the condition in Theorem \ref{ncon}. Consequently, the first two conditions proposed by Hujdurovi\'c et al. (\cite[Theorem 3.2,~Proposition 3.7]{HMM2021}) along with the condition in Theorem \ref{ncon} can be unified and optimized as follows:
\begin{theorem}
\label{hmmtype}
Let $\Gamma=\Cay(\mathbb{Z}_{2m}, S)$ be a circulant graph.
Let $\Gamma=\Cay(\mathbb{Z}_{2m}, S)$ be a circulant graph.
Then $\Gamma$ is unstable if any of the following conditions is true:
\begin{enumerate}
  \item $(S\setminus K_o)+H=S\setminus K_o$ where $K$ is an even order subgroup of $\mathbb{Z}_{2m}$, $K_o=K\setminus 2K$ and $H$ is an odd order subgroup of $K$;
  \item $\Gamma\cong\Cay(\mathbb{Z}_{2m}, S+m)$;
  \item $\Cay\big(\mathbb{Z}_{2m},(S\setminus\{m\})+m\big)$ has an automorphism fixing $0$ but moving $m$.
\end{enumerate}
\end{theorem}

\medskip
The third condition of Hujdurovi\'c et al.
(\cite[Proposition 3.12]{HMM2021}) provides a method for constructing large unstable circulant graph from smaller ones. We generalize this construction to Proposition \ref{preserving} and accordingly categorize unstable circulant graphs $\Gamma=\Cay(\mathbb{Z}_{2m}, S)$ into two distinct types: those for which there exists a TF-morphism $(\alpha,\beta)$ such that both $\alpha$ and $\beta$ preserve $2\mathbb{Z}_{2m}$, and those for which no such TF-morphism exists. For the former, we observe that they satisfy the condition of Proposition \ref{preserving}. For the latter, we propose three conjectures that aid in classification.

\medskip
The remainder of the paper is structured as follows. Section \ref{sec2} presents preliminary results used in subsequent sections. Section \ref{sec3} establishes several fundamental results on graph products.  Section \ref{sec4} contains the proofs of Theorems \ref{directp} and \ref{dbundle}, while Section \ref{sec5} is devoted to the proof of Theorem \ref{others}. In Section \ref{sec6}, we employ graph products to construct nontrivially unstable circulant graphs. In Section \ref{sec7}, we investigate the stability of circulant graphs and provide the proof of Theorem \ref{ncon}. Finally, in Section \ref{sec8}, we categorize unstable circulant graphs into two distinct types and propose a classification framework based on this categorization.

\section{Preliminaries}
\label{sec2}

Employing the language of symmetric $(0, 1)$-matrices, Maru\v si\v c et al. \cite{MSZ1989} established a necessary and sufficient condition for stability of connected non-bipartite graphs. In view of TF-morphisms, this result can be reformulated as follows:
\begin{lem}
\label{TF}
A connected non-bipartite graph is unstable if and only if it has a nontrivial TF-morphism.
\end{lem}
While this restatement was first proposed by Lauri et al. \cite{LMS2014}, the non-bipartiteness condition was inadvertently overlooked. We emphasize that this condition is crucial, as its omission compromises the validity of the criterion. For instance, $K_2$ is unstable yet admits no TF-morphism.

\medskip
The framework of TF-morphisms emphasizes the intrinsic structural properties of the graph itself, providing a streamlined method for analyzing stability. We now apply this framework to establish the following lemma.

\begin{lem}
\label{product}
Let $\Gamma$ and $\Sigma$ be two connected graphs, and let $L$ and $R$ be symmetric relations on $V(\Gamma)$ and $V(\Sigma)$, respectively. Let $\Gamma\ast\Sigma$ denote the graph with vertex set $V(\Gamma)\times V(\Sigma)$ such that two vertices $(a,x)\sim_{\Gamma\ast\Sigma}(b,y)$ if and only if $a\sim_L b$ and $x\sim_R y$.
If either $\Gamma$ is unstable with $E(\Gamma)=L$ or $\Sigma$ is unstable with $E(\Sigma)=R$, then $\Gamma\ast\Sigma$ is unstable.
\end{lem}
\begin{proof}
Without loss of any generality, we assume that $\Sigma$ is unstable and $R= E(\Sigma)$.

If $\Sigma$ is a bipartite graph with bipartition $\{U,W\}$, then it is obvious that $\Gamma\ast\Sigma$ is a bipartite graph with bipartition $\{V(\Gamma)\times U,V(\Gamma)\times W\}$. Since $\Sigma$ is connected and unstable, it has a nontrivial automorphism, say $\sigma$. It is straightforward to check that the permutation $(1,\sigma)$ on $V(\Gamma)\times V(\Sigma)$ is a nontrivial automorphism of $\Gamma\ast\Sigma$. Thus $\Gamma\ast\Sigma$ is unstable.

Now we assume that $\Sigma$ is non-bipartite.  Since $\Sigma$ is unstable, by Lemma \ref{TF} there exists a TF-morphism $(\alpha,\beta)$ of $\Sigma$ such that  $\alpha\ne\beta$. Since $R= E(\Sigma)$, we have $x\sim_R y\Longleftrightarrow x\sim_{\Sigma} y$. Consider the two permutations $(1,\alpha)$ and  $(1,\beta)$ on $V(\Gamma)\times V(\Sigma)$. Then $(1,\alpha)\ne(1,\beta)$ and
\begin{align*}
  (a,x)\sim_{\Gamma\ast\Sigma}(b,y)\Longleftrightarrow&
  a\sim_{L}b~\mbox{and}~x\sim_{R}y \\
  \Longleftrightarrow&
  a\sim_{L}b~\mbox{and}~x\sim_{\Sigma}y \\
  \Longleftrightarrow& a\sim_{L}b~ \mbox{and}~x^{\alpha}\sim_{\Sigma}y^{\beta}\\
  \Longleftrightarrow& a\sim_{L}b~ \mbox{and}~x^{\alpha}\sim_{R}y^{\beta}\\
  \Longleftrightarrow& (a,x^{\alpha})\sim_{\Gamma\ast\Sigma} (b,y^{\beta})\\
  \Longleftrightarrow& (a,x)^{(1,\alpha)}\sim_{\Gamma\ast\Sigma} (b,y)^{(1,\beta)}.
\end{align*}
Therefore $((1,\alpha),(1,\beta))$ is a nontrivial TF-morphism of $\Gamma\ast\Sigma$. By Lemma \ref{TF}, $\Gamma\ast\Sigma$ is unstable.
\end{proof}
\begin{lem}[{\cite[Proposition 8.10]{HIK2011}}]
\label{sdc}
Let $\Gamma$ and $\Sigma$ be two $R$-thin graphs without isolated vertices. Then
$S(\Gamma\times\Sigma)=S(\Gamma)\Box S(\Sigma)$.
\end{lem}

The following lemma is extracted from \cite[Lemma 4.2]{QXZ2019}.
\begin{lem}
\label{tfsk}
Let $\Gamma$ be a graph with a TF-morphism $(\alpha,\beta)$. Then $\alpha,\beta\in\Aut(S(\Gamma))$.
\end{lem}
For a TFS-morphism $(\alpha,\beta)$ of $\Gamma$, the composition $\alpha\beta^{-1}$ is clearly a \emph{derangement} (see \cite[Definition 1.1.1]{BG2016}), meaning it fixes no vertex of $\Gamma$. The following lemma describes the relationship between TF-morphisms and TFS-morphisms.
\begin{lem}
\label{tfaa}
Let $\alpha$ and $\beta$ be two permutations on the vertex set $V(\Gamma)$ of a graph $\Gamma$. Then $(\alpha,\beta)$ is a TFS-morphism of $\Gamma$ if and only if $\alpha\beta^{-1}$ is a derangement and  $(\alpha,\beta)$ is a TF-morphism of $\overline{\Gamma}$.
\end{lem}
\begin{proof}
First we prove the necessity.
Let $(\alpha,\beta)$ be a TFS-morphism of $\Gamma$. For $u\in V(\Gamma)$, since $u^\alpha\sim_{\Gamma}u^\beta$, we have $u^{\alpha\beta^{-1}}\neq u$ and therefore $\alpha\beta^{-1}$ is a derangement.
By Remark \ref{rtfs} (ii), we have
\begin{equation*}
u\sim_{\Gamma}v\Longleftrightarrow u^\alpha=v^\beta~\mbox{or}~
u^\alpha\sim_{\Gamma}v^\beta
\end{equation*}
and it follows that
\begin{equation*}
u\sim_{\overline{\Gamma}}v\Longleftrightarrow
u^\alpha\sim_{\overline{\Gamma}}v^\beta.
\end{equation*}
for any two distinct vertices $u,v\in V(\Gamma)$.
Thus $(\alpha,\beta)$ is a TF-morphism of $\overline{\Gamma}$.

Now we prove the sufficiency. Let $(\alpha,\beta)$ be a TF-morphism of $\overline{\Gamma}$ and $\alpha\beta^{-1}$ be a derangement. By the definition of TF-morphism, we have that $(\alpha,\beta)$ acts as a permutation on the set
$\{(u,v)\mid u\sim_{\overline{\Gamma}}v,~u,v\in V(\Gamma)\}$. Thus either $u^{\alpha}=v^{\beta}$ or $u^{\alpha}\sim_{\Gamma}v^{\beta}$ for any two vertices $u$ and $v$ not adjacent in $\overline{\Gamma}$. In particular, $u\sim_{\Gamma}v$ implies $u^{\alpha}=v^{\beta}$ or $u^{\alpha}\sim_{\Gamma}v^{\beta}$. Since $\alpha\beta^{-1}$ is a derangement, we have $u^{\alpha}\neq u^{\beta}$ and it follows that $u^{\alpha}\sim_{\Gamma}u^{\beta}$.
Therefore $(\alpha,\beta)$ is a TFS-morphism of $\Gamma$.
\end{proof}

\section{On graph products}
\label{sec3}

In this section, we present some basic results on graph products. The subsequent lemmas are straightforward to establish, and most are well known (see, e.g., \cite{HIK2011}). Proofs are provided only for results not previously established in the literature and requiring non-trivial argumentation.

\begin{lem}
\label{cnthdp}
Let $\Gamma$ and $\Sigma$ be two graphs. Then
\begin{enumerate}
\item $\Gamma\times\Sigma$ is connected if and only if both $\Gamma$ and $\Sigma$ are connected and at least one of them is non-bipartite;
\item $\Gamma\times\Sigma$ is non-bipartite if and only if both $\Gamma$ and $\Sigma$ are non-bipartite;
\item $\Gamma\times\Sigma$ is $R$-thin if and only if both $\Gamma$ and $\Sigma$ are $R$-thin.
\end{enumerate}
\end{lem}
\begin{lem}
\label{cnthcp}
Let $\Gamma$ and $\Sigma$ be two graphs without isolated vertices. Then
\begin{enumerate}
\item $\Gamma\Box\Sigma$ is connected if and only if both $\Gamma$ and $\Sigma$ are connected;
\item $\Gamma\Box\Sigma$ is non-bipartite if and only if at least one of $\Gamma$ and $\Sigma$ is non-bipartite;
\item $\Gamma\Box\Sigma$ is $R$-thick if and only if it contains a connected component which is a $4$-cycle.
\end{enumerate}
\end{lem}
\begin{proof}
(iii)~The sufficiency is obviously true. Now we prove the necessity. Assume that $\Gamma\Box\Sigma$ is $R$-thick. Set $\Lambda:=\Gamma\Box\Sigma$ and let $(a,x)$ and $(b,y)$ be two distinct vertices of $\Lambda$ such that $\Lambda(a,x)=\Lambda(b,y)$.
Since
\begin{equation*}
\Lambda(a,x)=\left(\Gamma(a)\times \{x\}\right)\cup \left(\{a\}\times\Sigma(x)\right)~\mbox{and}~
\Lambda(b,y)=\left(\Gamma(b)\times \{y\}\right)\cup \left(\{b\}\times\Sigma(y)\right),
\end{equation*}
we have
\begin{equation}\label{axby}
\left(\Gamma(a)\times \{x\}\right)\cup \left(\{a\}\times\Sigma(x)\right)=\left(\Gamma(b)\times \{y\}\right)\cup \left(\{b\}\times\Sigma(y)\right).
\end{equation}
Noting that $a\notin\Gamma(a)$, we have
$\left(\{a\}\times\Sigma(x)\right)\cap\left(\Gamma(a)\times \{y\}\right)=\emptyset$.
Therefore, if $x\neq y$, then
$\Lambda(a,x)\cap\left(\Gamma(a)\times \{y\}\right)=\emptyset$ which leads to $\Lambda(a,x)\neq\Lambda(a,y)$. Similarly, $\Lambda(a,x)\neq\Lambda(b,x)$ whenever $a\neq b$. Thus $a\neq b$ and $x\neq y$. It follows that
\begin{equation*}
\left(\Gamma(a)\times \{x\}\right)\cap \left(\Gamma(b)\times \{y\}\right)=\emptyset~\mbox{and}~ \left(\{a\}\times\Sigma(x)\right)\cap \left(\{b\}\times\Sigma(y)\right)=\emptyset.
\end{equation*}
Combining the equation \ref{axby}, we have
\begin{equation*}
\Gamma(a)\times \{x\}=\{b\}\times\Sigma(y)~\mbox{and}~
\{a\}\times\Sigma(x)=\Gamma(b)\times \{y\}.
\end{equation*}
Therefore
\begin{equation*}
\Gamma(a)=\{b\},~\{x\}=\Sigma(y),\{a\}=\Gamma(b)~\mbox{and}~
\Sigma(x)=\{y\}.
\end{equation*}
It follows that the four vertices $(a,x),(a,y),(b,x),(b,y)$ induce a connected component which is specifically a $4$-cycle.
\end{proof}
\begin{coro}
Let $\Gamma$ and $\Sigma$ be two connected graphs without isolated vertices. Then $\Gamma\Box\Sigma$ is $R$-thick if and only if both $\Gamma$ and $\Sigma$ are isomorphic to $K_2$.
\end{coro}
\begin{lem}
\label{cnthsp}
Let $\Gamma$ and $\Sigma$ be two graphs without isolated vertices. Then
\begin{enumerate}
\item $\Gamma\boxtimes\Sigma$ is connected if and only if both $\Gamma$ and $\Sigma$ are connected;
\item $\Gamma\boxtimes\Sigma$ is non-bipartite;
\item $\Gamma\boxtimes\Sigma$ is $R$-thin.
\end{enumerate}
\end{lem}
\begin{proof}
(iii)~Write $\Lambda=\Gamma\boxtimes\Sigma$. Let $a,b$ be two distinct vertices of $\Gamma$ and $x,y$ be two distinct vertices of $\Sigma$. We will complete the proof by confirming  two claims as follows.

\textsf{Claim 1.} $\Lambda(a,x)\neq \Lambda(a,y)$ and $\Lambda(a,x)\neq \Lambda(b,x)$.

If $x\sim_{\Sigma}y$, then $(a,x)\in \Lambda(a,y)$ but $(a,x)\notin \Lambda(a,x)$. Therefore $\Lambda(a,x)\neq \Lambda(a,y)$. Now we assume $x\nsim_{\Sigma}y$. Since $\Gamma$ has no isolated vertices, there exists $c\in \Gamma(a)$. Clearly, $(c,x)\in \Lambda(a,x)$ but $(c,x)\notin \Lambda(a,y)$. Thus we also have $\Lambda(a,x)\neq \Lambda(a,y)$.

By analogous reasoning, we conclude $\Lambda(a,x)\neq \Lambda(b,x)$.

\textsf{Claim 2.} $\Lambda(a,x)\neq \Lambda(b,y)$.

If $\Gamma[a]\neq\Gamma[b]$, then $\Gamma[a]\setminus\Gamma[b]\ne\emptyset$ or $\Gamma[b]\setminus\Gamma[a]\ne\emptyset$. Without loss of generality, assume $\Gamma[a]\setminus\Gamma[b]\ne\emptyset$. Let $c\in \Gamma[a]\setminus\Gamma[b]$. Then $\{c\}\times\Sigma(x)\subseteq \Lambda(a,x)$ but $\left(\{c\}\times\Sigma(x)\right)\cap \Lambda(b,y)=\emptyset$. Since $\Sigma$ has no isolated vertices, we have $\Sigma(x)\ne\emptyset$. Thus $\Lambda(a,x)\neq \Lambda(b,x)$. Similarly, $\Lambda(a,x)\neq \Lambda(b,x)$ whenever
$\Sigma[x]\neq\Sigma[y]$. Now we assume $\Gamma[a]=\Gamma[b]$ and $\Sigma[x]=\Sigma[y]$. Then $\Gamma[a]\times\Sigma[x]=\Gamma[b]\times\Sigma[y]$. Note that $(a,x)\neq (b,y)$. Since
\begin{equation*}
\Lambda(a,x)=\Gamma[a]\times\Sigma[x]\setminus\{(a,x)\}
~\mbox{and}~ \Lambda(b,y)=\Gamma[b]\times\Sigma[y]\setminus\{(b,y)\},
\end{equation*}
we get $\Lambda(a,x)\neq \Lambda(b,y)$.
\end{proof}
\begin{lem}
\label{cnbthss}
Let $\Gamma$ and $\Sigma$ be two graphs. Then
\begin{enumerate}
\item $\Gamma\ltimes\Sigma$ is connected if and only if both $\Gamma$ and $\Sigma$ are connected;
\item $\Gamma\ltimes\Sigma$ is non-bipartite if and only if $\Sigma$ is non-bipartite;
\item $\Gamma\ltimes\Sigma$ is $R$-thin if and only if $\Sigma$ is $R$-thin and $\Gamma[a]\ne\Gamma[b]$ for each pair of distinct vertices $a$ and $b$ of $\Gamma$.
\end{enumerate}
\end{lem}
\begin{proof}
(iii)~Write $\Lambda=\Gamma\ltimes\Sigma$. By the definition of $\Gamma\ltimes\Sigma$, we have that  $\Lambda(a,x)=\Lambda[a]\times\Lambda(x)$ for every vertex of $\Lambda$. Therefore two vertices $(a,x)$ and
$(b,y)$ of $\Gamma\ltimes\Sigma$ have the same neighbourhood if and only if $\Lambda[a]=\Lambda[b]$ and $\Lambda(x)=\Lambda(y)$.
Let $(a,x)\ne (b,y)$. Then $a\ne b$ or $x\ne y$.

First assume that $\Sigma$ is $R$-thin, and $\Gamma[a]\ne\Gamma[b]$ whenever $a\ne b$.  If $a=b$, then $x\ne y$. Since $\Sigma$ is $R$-thin, we have $\Sigma(x)\ne\Sigma(y)$ and it follows that $\Lambda(a,x)\ne\Lambda(b,y)$. If $a\ne b$, then $\Gamma[a]\ne\Gamma[b]$ and therefore $\Lambda(a,x)\ne\Lambda(b,y)$. Thus we always have $\Lambda(a,x)\ne\Lambda(b,y)$. This proves the sufficiency.

Now assume that $\Gamma\ltimes\Sigma$ is $R$-thin. Then $\Lambda(a,x)\ne\Lambda(a,y)$ whenever $x\ne y$ and $\Lambda(a,x)\ne\Lambda(b,x)$ whenever $a\ne b$. It follows that $\Sigma$ is $R$-thin and $\Gamma[a]\ne\Gamma[b]$ for each pair of distinct vertices $a$ and $b$ of $\Gamma$. This proves the necessity.
\end{proof}
\begin{lem}
\label{lpthin}
Let $\Gamma$ and $\Sigma$ be two graphs. Then
\begin{enumerate}
\item $\Gamma[\Sigma]$ is connected if and only if $\Gamma$ is connected;
\item $\Gamma[\Sigma]$ is bipartite if and only if both $\Gamma$ and $\Sigma$ are bipartite;
\item $\Gamma[\Sigma]$ is $R$-thin if and only if $\Sigma$ is $R$-thin.
\end{enumerate}
\end{lem}
\begin{proof}
(iii) For convenience, set $\Lambda:=\Gamma[\Sigma]$. Then
\begin{equation}\label{neighbor}
\Lambda(a,x)=\left(\Gamma(a)\times V(\Sigma)\right)\cup \left(\{a\}\times\Sigma(x)\right)
\end{equation}
for every $(a,x)\in V(\Lambda)$.

We first prove the necessity. Suppose
$\Sigma$ is $R$-thick. Then there exist distinct vertices $x,y\in V(\Sigma)$ such that $\Sigma(x)=\Sigma(y)$. By equation \eqref{neighbor}, this implies $\Lambda(a,x)=\Lambda(a,y)$ for every $a\in V(\Gamma)$, and consequently, $\Lambda$ is $R$-thick. It follows that $\Sigma$ must be $R$-thin whenever $\Lambda$ is $R$-thin.

Now we prove the sufficiency. Assume that $\Sigma$ is $R$-thin. Let $(a,x)$ and $(b,y)$ be distinct vertices of $\Lambda$. Then $\Sigma(x)\ne\Sigma(y)$.
We analyze three cases using equation \eqref{neighbor}:

If $a=b$, then $\Gamma(a)\times V(\Sigma)=\Gamma(b)\times V(\Sigma)$. Since $\Sigma(x)\ne\Sigma(y)$, we get $\{a\}\times\Sigma(x)\ne\{b\}\times\Sigma(y)$. Thus $\Lambda(a,x)\ne\Lambda(a,y)$.

If $a\sim_{\Gamma}b$, then $(b,y)\in \Lambda(a,x)$. However,
$(b,y)\notin\Lambda(b,y)$. Therefore $\Lambda(a,x)\ne\Lambda(b,y)$.

If $a\neq b$ and $a\nsim_{\Gamma}b$, then  $\{b\}\times\Sigma(y)\cap \Lambda(a,x)=\emptyset$ and hence $\Lambda(a,x)\ne\Lambda(a,y)$.

In all cases, we have $\Lambda(a,x)\ne\Lambda(a,y)$. Thus $\Lambda$ is $R$-thin.
\end{proof}
\section{Proofs of Theorems \ref{directp} and \ref{dbundle}}
\label{sec4}

In this section, we prove Theorems \ref{directp} and \ref{dbundle}. For clarity, we will restate both theorems before proceeding with their proofs.
\begin{directp}
Let $\Gamma$ and $\Sigma$ be two graphs. We have
\begin{enumerate}
  \item if $\Gamma$ or $\Sigma$ is unstable, then $\Gamma\times\Sigma$ is unstable;
  \item if $S(\Sigma)$ and $S(\Gamma)$ are coprime with respect to Cartesian product and $\Gamma$ and $\Sigma$ are both stable, then $\Gamma\times\Sigma$ is stable;
  \item if $S(\Gamma)$ and $S(\Sigma)$ are coprime with respect to Cartesian product and $\Gamma\times\Sigma$ is nontrivially unstable, then either $\Gamma$ or  $\Sigma$ is nontrivially unstable.
\end{enumerate}
\end{directp}
\begin{proof}

(i) It is obvious that the direct product $\Gamma\times\Sigma$ is exactly the graph $\Gamma\ast\Sigma$ defined in Lemma \ref{product} with $L=E(\Gamma)$ and $R=E(\Sigma)$. Therefore $\Gamma\times\Sigma$ is unstable if one of $\Gamma$ and $\Sigma$ is unstable.

(ii)
Since $\Gamma$ and $\Sigma$ are stable, they are connected, non-bipartite, and \(R\)-thin, and, in particular, they contain no isolated vertices.
Let $(\alpha,\beta)$ be an arbitrary TF-morphism of $\Gamma\times\Sigma$. By Lemma \ref{TF}, it suffices to show that  $(\alpha,\beta)$ is trivial, that is, $\alpha=\beta$. By Lemma \ref{tfsk}, we have $\alpha,\beta\in\Aut(S(\Gamma\times\Sigma))$. By Lemma \ref{sdc}, we get $S(\Gamma\times\Sigma)=S(\Gamma)\Box S(\Sigma)$ and
therefore $\alpha,\beta\in\Aut(S(\Gamma)\Box S(\Sigma))$. Since $S(\Sigma)$ and $S(\Gamma)$ are coprime with respect to Cartesian product, we have $\Aut(S(\Gamma)\Box S(\Sigma))=\Aut(S(\Gamma))\times \Aut(S(\Sigma))$ and it follows that $\alpha,\beta\in\Aut(S(\Gamma))\times \Aut(S(\Sigma))$. Thus we can write $\alpha=(\alpha_1,\alpha_2)$ and $\beta=(\beta_1,\beta_2)$ where $\alpha_1,\beta_1\in\Aut(S(\Gamma))$ and
$\alpha_2,\beta_2\in\Aut(S(\Sigma))$. In particular, both  $\alpha_1$ and $\beta_1$ are permutations on $V(\Gamma)$ and
both $\alpha_2$ and $\beta_2$ are permutations on $V(\Sigma)$.
Since $(\alpha,\beta)$ is a TF-morphism of $\Gamma\times\Sigma$, we have
\begin{align*}
  (a,x)\sim_{\Gamma\times\Sigma}
  (b,y)\Longleftrightarrow&(a,x)^{\alpha}
  \sim_{\Gamma\times\Sigma}
  (b,y)^{\beta}\\
  \Longleftrightarrow&(a^{\alpha_{1}},x^{\alpha_{2}})
  \sim_{\Gamma\times\Sigma}
  (b^{\beta_{1}},y^{\beta_{2}})\\
  \Longleftrightarrow&a^{\alpha_{1}}
  \sim_{\Gamma}b^{\beta_{1}}
  ~\mbox{and}~
  x^{\alpha_{2}}\sim_{\Sigma}y^{\beta_{2}}
\end{align*}
for every pair of elements $(a,x)$ and $(b,y)$ of $\Gamma\times\Sigma$. Thus $a\sim_{\Gamma}b$ implies $b^{\alpha_{1}}\sim_{\Gamma}b^{\beta_{1}}$, that is, $(\alpha_1,\beta_1)$ is TF-morphism of $\Gamma$. Since $\Gamma$ is stable, we have $\alpha_1=\beta_1$. Similarly, $\alpha_2=\beta_2$. It follows that $\alpha=\beta$ and therefore $\Gamma\times\Sigma$ is stable.

(iii) Since $\Gamma\times\Sigma$ is nontrivially unstable, $\Gamma\times\Sigma$ is an unstable graph which is connected, $R$-thin and non-bipartite. By the conclusion of (ii), either $\Gamma$ or $\Sigma$ is unstable. By Lemma \ref{cnthdp}, both $\Gamma$ and $\Sigma$ are connected, $R$-thin and non-bipartite. Thus either $\Gamma$ or $\Sigma$ is nontrivially unstable.
\end{proof}
By using Theorem \ref{directp}, we prove the following result on the stability of graph pairs.
\begin{coro}
\label{direct}
Let $\Gamma$ and $\Sigma$ be two non-bipartite stable graphs whose orders are coprime. Then the graph pair $(\Gamma,\Sigma\times K_2)$ is stable.
\end{coro}
\begin{proof}
Since both $\Gamma$ and $\Sigma$ are stable, they are connected and $R$-thin. Since the orders of $\Gamma$ and $\Sigma$ are coprime, we conclude that $\Gamma$ and $\Sigma$ are coprime with respect to direct product and $S(\Sigma)$ and $S(\Gamma)$ are coprime with respect to Cartesian product. It follows from Theorem \ref{directp} (ii) that $\Gamma\times\Sigma$ is stable. Therefore $\Aut(\Gamma\times\Sigma\times K_2)=\Aut(\Gamma\times\Sigma)\times \mathbb{Z}_2$. Since a pair of two connected non-bipartite graphs is stable if the two graphs are both $R$-thin and they are coprime with respect to direct product \cite[Theorem 8.18]{HIK2011}, we conclude that the pair $(\Gamma,\Sigma)$ is stable and it follows that
\begin{equation*}
  \Aut(\Gamma\times\Sigma\times K_2)=\Aut(\Gamma\times\Sigma)\times \mathbb{Z}_2=\Aut(\Gamma)\times\Aut(\Sigma)\times \mathbb{Z}_2.
\end{equation*}
Since $\Sigma$ is stable, we have $\Aut(\Sigma\times K_2)=\Aut(\Sigma)\times \mathbb{Z}_2$ and hence
\begin{equation*}
\Aut(\Gamma\times\Sigma\times K_2)=\Aut(\Gamma)\times\Aut(\Sigma\times K_2).
\end{equation*}
Therefore $(\Gamma,\Sigma\times K_2)$ is stable.
\end{proof}

\begin{dbundle}
Let $\Gamma$ and $\Sigma$ be two graphs and $p$ be a mapping from $V(\Gamma)\times V(\Gamma)$ to $\Aut(\Sigma)$ such that $p(a,b)=p(b,a)^{-1}$ for every $(a,b)\in V(\Gamma)\times V(\Gamma)$. Then  $\Gamma\times^{p}\Sigma$ is unstable if one of the following two statements holds:
\begin{enumerate}
  \item $\Gamma$ has a nontrivial TF-morphism $(\alpha,\beta)$ such that
  $p(a^{\alpha},b^{\beta})=p(a,b)$ for every $(a,b)\in V(\Gamma)\times V(\Gamma)$;
  \item there is a TF-morphism $(\alpha,\beta)$ of $\Sigma$ and a permutation $\theta$ on $V(\Sigma)$ such that $\alpha\neq\theta$ and $\theta p(a,b)^{-1}=p(a,b)^{-1}\beta$ for every $(a,b)\in V(\Gamma)\times V(\Gamma)$.
\end{enumerate}
\end{dbundle}
\begin{proof}
(i) Let $(\alpha,\beta)$ be a nontrivial TF-morphism of $\Gamma$ such that $p(a^{\alpha},b^{\beta})=p(a,b)$ for every $(a,b)\in V(\Gamma)\times V(\Gamma)$. Then $\alpha\ne\beta$ and
\begin{align*}
  (a,x)\sim_{\Gamma\times^{p}\Sigma}(b,y)\Longleftrightarrow&
  a\sim_{\Gamma}b~\mbox{and}~x\sim_{\Sigma}y^{p(a,b)^{-1}} \\
  \Longleftrightarrow& a^{\alpha}\sim_{\Gamma}b^{\beta}~ \mbox{and}~x\sim_{\Sigma}y^{p(a^{\alpha},b^{\beta})^{-1}}\\
  \Longleftrightarrow& (a^{\alpha},x)\sim_{\Gamma\times^{p}\Sigma} (b^{\beta},y)\\
  \Longleftrightarrow& (a,x)^{(\alpha, 1)}\sim_{\Gamma\times^{p}\Sigma} (b,y)^{(\beta, 1)}.
\end{align*}
Therefore $((\alpha, 1),(\beta, 1))$ is a nontrivial TF-morphism of $\Gamma\times\Sigma$.
By Lemma \ref{TF}, $\Gamma\times^{p}\Sigma$ is unstable.

(ii) Since $(\alpha,\beta)$ is a TF-morphism of $\Sigma$ and $\theta p(a,b)^{-1}=p(a,b)^{-1}\beta$ for every $(a,b)\in V(\Gamma)\times V(\Gamma)$, we have
\begin{align*}
  (a,x)\sim_{\Gamma\times^{p}\Sigma}(b,y)\Longleftrightarrow&
  a\sim_{\Gamma}b~\mbox{and}~x\sim_{\Sigma}y^{p(a,b)^{-1}} \\
  \Longleftrightarrow& a\sim_{\Gamma}b~ \mbox{and}~x^{\alpha}\sim_{\Sigma}y^{p(a,b)^{-1}\beta}\\
   \Longleftrightarrow& a\sim_{\Gamma}b~ \mbox{and}~x^{\alpha}\sim_{\Sigma}y^{\theta p(a,b)^{-1}}\\
  \Longleftrightarrow& (a,x^{\alpha})\sim_{\Gamma\times^{p}\Sigma} (b,y^{\theta})\\
  \Longleftrightarrow& (a,x)^{(1,\alpha)}\sim_{\Gamma\times^{p}\Sigma} (b,y)^{(1,\theta)}.
\end{align*}
Therefore $((1,\alpha),(1,\theta))$ is a TF-morphism of $\Gamma\times\Sigma$. Since $\alpha\ne\theta$, we have $(1,\alpha)\ne(1,\theta)$.
By Lemma \ref{TF}, $\Gamma\times^{p}\Sigma$ is unstable.
\end{proof}

\section{Proof of Theorem \ref{others}}
\label{sec5}

The proof of Theorem \ref{others} is structured around the following five propositions.
\begin{pro}
\label{car}
Let $\Gamma$ and $\Sigma$ be two graphs. If one of $\Gamma$ and $\Sigma$ is bipartite and the other has a nontrivial TF-morphism, then $\Gamma\Box\Sigma$ is unstable.
\end{pro}
\begin{proof}
Without loss of generality, assume that $\Gamma$ is bipartite and $\Sigma$ has a nontrivial TF-morphism.
Let $(\alpha,\beta)$ be a nontrivial TF-morphism of $\Sigma$. Then   $\alpha\ne\beta$. Define two permutations $\kappa$ and  $\tau$ on $V(\Gamma)\times V(\Sigma)$ as follows:
\begin{equation*}
 (a,x)^{\kappa}=\left\{
                   \begin{array}{ll}
                     (a,x^{\alpha}), & \hbox{if}~a\in U; \\
                     (a,x^{\beta}), & \hbox{if}~a\in W
                   \end{array}
                 \right.
~\mbox{and}~
(a,x)^{\tau}=\left\{
                   \begin{array}{ll}
                     (a,x^{\beta}), & \hbox{if}~a\in U; \\
                     (a,x^{\alpha}), & \hbox{if}~a\in W
                   \end{array}
                 \right.
\end{equation*}
where $\{U,W\}$ is a bipartition of $\Gamma$. Clearly, $\kappa\neq \tau$.
Let $a\in U$, $b\in W$ and $x,y\in V(\Sigma)$. Then
\begin{align*}
  (a,x)\sim_{\Gamma\Box\Sigma}(a,y)\Longleftrightarrow&
 x\sim_{\Sigma}y \\
  \Longleftrightarrow&
  x^{\alpha}\sim_{\Sigma}y^{\beta} \\
  \Longleftrightarrow& (a,x^{\alpha})\sim_{\Gamma\Box\Sigma} (a,y^{\beta})\\
  \Longleftrightarrow& (a,x)^{\kappa}\sim_{\Gamma\Box\Sigma} (a,y)^{\tau}
\end{align*}
and
\begin{align*}
  (a,x)\sim_{\Gamma\Box\Sigma}(b,x)\Longleftrightarrow&
 a\sim_{\Gamma}b \\
  \Longleftrightarrow&
   (a,x^{\alpha})\sim_{\Gamma\Box\Sigma} (b,x^{\alpha})\\
  \Longleftrightarrow& (a,x)^{\kappa}\sim_{\Gamma\Box\Sigma} (b,x)^{\tau}.
\end{align*}
Similarly,
\begin{equation*}
(b,x)\sim_{\Gamma\Box\Sigma}(b,y)\Longleftrightarrow (b,x)^{\kappa}\sim_{\Gamma\Box\Sigma} (b,y)^{\tau}
\end{equation*}
and
\begin{equation*} (b,x)\sim_{\Gamma\Box\Sigma}(a,x)\Longleftrightarrow (b,x)^{\kappa}\sim_{\Gamma\Box\Sigma} (a,x)^{\tau}.
\end{equation*}
Therefore $(\kappa,\tau)$ is a nontrivial TF-morphism of $\Gamma\Box\Sigma$. By Lemma \ref{TF}, $\Gamma\Box\Sigma$ is unstable.
\end{proof}
\begin{pro}
\label{str}
Let $\Gamma$ and $\Sigma$ be two connected graphs. If one of $\Gamma$ and $\Sigma$ is bipartite and the other is $R$-thick, then $\Gamma\boxtimes\Sigma$ is unstable.
\end{pro}
\begin{proof}
Without loss of generality, assume that $\Gamma$ is bipartite and $\Sigma$ is $R$-thick. Let $\{U,W\}$ be the bipartition of $\Gamma$ and $z_1$ and $z_2$ be two vertices of $\Sigma$ such that $x\sim_{\Sigma}z_1\Longleftrightarrow x\sim_{\Sigma}z_2$ for every $x\in V(\Sigma)$. Recall that graphs in this paper are assumed to be loopless. Thus we have $z_1\nsim_{\Sigma}z_2$. Define two permutations $\alpha$ and $\beta$ on $V(\Gamma)\times V(\Sigma)$ as follows:
\begin{equation*}
 (a,x)^{\alpha}=\left\{
                   \begin{array}{lll}
                     (a,z_2), & \hbox{if}~a\in U~\hbox{and}~x=z_1; \\
                     (a,z_1), & \hbox{if}~a\in U~\hbox{and}~x=z_2; \\
                     (a,x), & \hbox{otherwise}
                   \end{array}
                 \right.
\end{equation*}
and
\begin{equation*}
 (a,x)^{\beta}=\left\{
                   \begin{array}{lll}
                     (a,z_2), & \hbox{if}~a\in W~\hbox{and}~x=z_1; \\
                     (a,z_1), & \hbox{if}~a\in W~\hbox{and}~x=z_2; \\
                     (a,x), & \hbox{otherwise}.
                   \end{array}
                 \right.
\end{equation*}

Let $(a,x)$ and $(b,y)$ be two adjacent vertices of $\Gamma\boxtimes\Sigma$ where $a,b\in V(\Gamma)$ and $x,y\in V(\Sigma)$. Then we have
\begin{equation*}
  \left\{
    \begin{array}{ll}
     a=b~\mbox{and}~x\sim_{\Sigma}y,~\mbox{or} \\
     a\sim_{\Gamma}b~\mbox{and}~x\sim_{\Sigma}y,~\mbox{or}\\
     a\sim_{\Gamma}b~\mbox{and}~x=y.
    \end{array}
  \right.
\end{equation*}
By the definition of $\alpha$ and $\beta$, we can set $(a,x)^{\alpha}=(a,x')$ and $(b,y)^{\beta}=(b,y')$.

If $\{x,y\}\cap \{z_1,z_2\}=\emptyset$,
then $(a,x)^{\alpha}=(a,x)$ and $(b,y)^{\beta}=(b,y)$.
If $x=y\in\{z_1,z_2\}$, then $a\sim_{\Gamma}b$ and $x'=y'\in\{z_1,z_2\}$. In either case, we have $(a,x)^{\alpha}\sim_{\Gamma\boxtimes\Sigma}(b,y)^{\beta}$.

Now we assume $\{x,y\}\cap \{z_1,z_2\}\ne\emptyset$ and $x\ne y$. Then $y\sim_{\Sigma}x$. Since $z_1\nsim_{\Sigma}z_2$, we get $\{x,y\}\ne \{z_1,z_2\}$. Therefore either $x\in \{z_1,z_2\}$ and $y\notin \{z_1,z_2\}$ or $x\notin \{z_1,z_2\}$ and $y\in \{z_1,z_2\}$. It follows that either $x'\in \{z_1,z_2\}$ and $y'=y$ or $x'=x$ and $y'\in \{z_1,z_2\}$. Since $z_1$ and $z_2$ have the same neighbourhood, we have that $y'\sim_{\Sigma}x'$. It follows that $(a,x)^{\alpha}\sim_{\Gamma\boxtimes\Sigma}(b,y)^{\beta}$.

Note that $\alpha\neq \beta$ and we have proved that $(a,x)^{\alpha}\sim_{\Gamma\boxtimes\Sigma}(b,y)^{\beta}$ for any two adjacent vertices $(a,x)$ and $(b,y)$ of $\Gamma\boxtimes\Sigma$.
Therefore $(\alpha,\beta)$ is a nontrivial TF-morphism of $\Gamma\boxtimes\Sigma$. By Lemma \ref{TF}, $\Gamma\boxtimes\Sigma$ is unstable.
\end{proof}
\begin{pro}
\label{tfaas}
Let $\Gamma$ and $\Sigma$ be two graphs, one of which has a TFS-morphism. Then $\overline{\Gamma\boxtimes\Sigma}$ is unstable.
\end{pro}
\begin{proof}
Without loss of generality, assume that $\Gamma$ has a TFS-morphism $(\alpha,\beta)$. Then $a^{\alpha}\sim_{\Gamma}a^{\beta}$ for every $a\in V(\Gamma)$ and $b^{\alpha}\sim_{\Gamma}c^{\beta}$ or $b^{\alpha}=c^{\beta}$ for each pair of adjacent vertices $b,c\in V(\Gamma)$. This implies that $(a^{\alpha},x)\sim_{\Gamma\boxtimes\Sigma}(a^{\beta},x)$ for every $a\in V(\Gamma)$ and either $(b^{\alpha},x)\sim_{\Gamma\boxtimes\Sigma}(c^{\beta},y)$ or $(b^{\alpha},x)=(c^{\beta},y)$ for each pair of adjacent vertices $(b,x),(c,y)\in V(\Gamma\boxtimes\Sigma)$. Thus $((\alpha,1),(\beta,1))$ is a TFS-morphism of $\Gamma\boxtimes\Sigma$. By Lemma \ref{tfaa}, $((\alpha,1),(\beta,1))$ is a nontrivial TF-morphism of $\overline{\Gamma\boxtimes\Sigma}$. Therefore $\overline{\Gamma\boxtimes\Sigma}$ is unstable.
\end{proof}
\begin{pro}
\label{semir}
Let $\Gamma$ and $\Sigma$ be two connected graphs which has at least two vertices. Then $\Gamma\ltimes\Sigma$ is unstable if one of the following two holds:
\begin{enumerate}
  \item $\Sigma$ is unstable;
  \item $\Gamma$ has a TFS-morphism.
\end{enumerate}
\end{pro}
\begin{proof}
(i) Note that $\Gamma\ltimes\Sigma$ is the graph defined in Lemma \ref{product} with $R=E(\Sigma)$ and $L=E(\Gamma)\cup \{(a,a)\mid a\in V(\Gamma)\}$. Since $\Sigma$ is unstable, Lemma \ref{product} implies that $\Gamma\ltimes\Sigma$ is unstable.

(ii) Let $(\alpha,\beta)$ be a TFS-morphism of $\Gamma$. We will confirm the instability of $\Gamma\ltimes\Sigma$ by showing that $((\alpha,1),(\beta,1))$ is a nontrivial TF-morphism of $\Gamma\ltimes\Sigma$. By the definition of TFS-morphism, we have that $a^\alpha\sim_{\Gamma}a^\beta$ for every $a\in V(\Gamma)$ and  $b^\alpha\sim_{\Gamma}c^\beta$ or $b^\alpha=c^\beta$ for each pair of adjacent vertices $b$ and $c$ of $\Gamma$.
Therefore $(\alpha,1)\neq (\beta,1)$ and
 \begin{align*}
  (a,x)\sim_{\Gamma\ltimes\Sigma}(a,y)\Longleftrightarrow&
  x\sim_{\Sigma}y \\ \Longleftrightarrow&
  a^\alpha\sim_{\Gamma}a^\beta~\mbox{and}~x\sim_{\Sigma}y \\
  \Longleftrightarrow&
  (a^\alpha,x)\sim_{\Gamma\ltimes\Sigma}(a^\beta,y) \\
  \Longleftrightarrow& (a,x)^{(\alpha,1)}\sim_{\Gamma\ltimes\Sigma} (a,y)^{(\beta,1)}
\end{align*}
for $(a,x),(a,y)\in V(\Gamma)\times V(\Sigma)$.
Now let $(b,x)$ and $(c,y)$ be a pair of vertices of $\Gamma\ltimes\Sigma$ with
$b\ne c$. Since $b^\alpha\sim_{\Sigma}c^\beta$ or $b^\alpha=c^\beta$ whenever $b\sim_{\Gamma}c$, we have
 \begin{align*}
  (b,x)\sim_{\Gamma\ltimes\Sigma}(c,y)\Longleftrightarrow&
  b\sim_{\Gamma}c~\mbox{and}~x\sim_{\Sigma}y \\
  \Longleftrightarrow&
  b^\alpha\sim_{\Gamma}c^\beta ~\mbox{or}~ b^\alpha=c^\beta,~\mbox{and}~x\sim_{\Sigma}y\\
  \Longleftrightarrow&
  (b^\alpha,x)\sim_{\Gamma\ltimes\Sigma}(c^\beta,y) \\
  \Longleftrightarrow& (b,x)^{(\alpha,1)}\sim_{\Gamma\ltimes\Sigma} (c,y)^{(\beta,1)}.
\end{align*}
Therefore $((\alpha,1),(\beta,1))$ is a nontrivial TF-morphism of $\Gamma\ltimes\Sigma$. By Lemma \ref{TF}, $\Gamma\ltimes\Sigma$ is unstable.
\end{proof}

\begin{pro}
\label{rlexi}
Let $\Gamma$ and $\Sigma$ be two graphs. If $\Sigma$ has a nontrivial TF-morphism, then $\Gamma[\Sigma]$ is unstable.
\end{pro}
\begin{proof}
Let $(\alpha,\beta)$ be a nontrivial TF-morphism of $\Sigma$. Then   $\alpha\ne\beta$. Consider the two permutations $(1,\alpha)$ and  $(1,\beta)$ on $V(\Gamma)\times V(\Sigma)$. Then $(1,\alpha)\ne(1,\beta)$ and
\begin{align*}
  (a,x)\sim_{\Gamma[\Sigma]}(b,y)\Longleftrightarrow&
  a\sim_{\Gamma}b,~\mbox{or}~a=b~\mbox{and}~x\sim_{\Sigma}y \\
  \Longleftrightarrow&
  a\sim_{\Gamma}b,~\mbox{or}~a=b~\mbox{and}~x^{\alpha}\sim_{\Sigma}y^{\beta} \\
  \Longleftrightarrow& (a,x^{\alpha})\sim_{\Gamma[\Sigma]} (b,y^{\beta})\\
  \Longleftrightarrow& (a,x)^{(1,\alpha)}\sim_{\Gamma[\Sigma]} (b,y)^{(1,\beta)}.
\end{align*}
Therefore $((1,\alpha),(1,\beta))$ is a nontrivial TF-morphism of $\Gamma[\Sigma]$. By Lemma \ref{TF}, $\Gamma[\Sigma]$ is unstable.
\end{proof}
\section{Constructions}
\label{sec6}
In this section, we use graph products to construct  nontrivially unstable circulant graphs.
Let $S$ be a subset of a group $G$ that does not contain the identity element and is closed to the inverse operation. The \emph{Cayley graph} $\Cay(G, S)$ of $G$ with the \emph{connection set} $S$ is defined to be the
graph with vertex set $G$ such that two elements $x, y$ are
adjacent if and only if $yx^{-1} \in S$ (or $y-x \in S$ if $G$ is an additive group).
In particular, we call $\Cay(G, S)$ a \emph{circulant graph} if $G$ is cyclic.

\medskip
We begin by stating two propositions, which follow directly from Theorems \ref{dbundle} (ii) and \ref{others} (iv), respectively.
\begin{pro}
\label{dihedral}
Let $\Gamma$ and $\Sigma$ be two graphs. If $\Sigma$ has an automorphism $\rho$ and an involutory automorphism $\delta$ such that $\delta^{-1}\rho\delta\ne\rho$, then $\Gamma\times^{p}\Sigma$ is unstable where $p(a,b)=\delta$ for all $(a,b)\in V(\Gamma)\times V(\Gamma)$.
\end{pro}
\begin{proof}
Let $\alpha=\beta=\rho$ and $\theta=\delta^{-1}\rho\delta$. Then $(\alpha,\beta)$ is a TF-morphism of $\Sigma$, $\alpha\ne\theta$ and $\theta p(a,b)^{-1}=\delta^{-1}\rho=p(a,b)^{-1}\beta$ for every $(a,b)\in V(\Gamma)\times V(\Gamma)$. By Theorem \ref{dbundle} (ii), $\Gamma\times^{p}\Sigma$ is unstable.
\end{proof}

Let $K_n$ denote the complete graph of order $n$.
Every pair $(\alpha, \beta)$ of permutations on $V(K_n)$, for which $\alpha\beta^{-1}$ is a derangement, clearly constitutes a TFS-morphism of $K_n$. However, Lemma \ref{cnbthss} (iii) implies that $K_n\ltimes\Sigma$ is $R$-thick.
Thus, $K_n$ cannot serve as $\Gamma$ in constructing a nontrivially unstable graph $\Gamma\ltimes\Sigma$.
On the other hand, removing a perfect matching from $K_{2n}$
produces a graph that is a suitable candidate for this purpose.
\begin{pro}
\label{k2n}
Let $n\geq 2$ and $\Gamma$ be a graph obtained from $K_{2n}$ by removing a perfect matching. Let $\Sigma$ be a connected, non-bipartite and $R$-thin graph. Then $\Gamma\ltimes\Sigma$ is nontrivially unstable.
\end{pro}
\begin{proof}
Let $K_{2n}$ be the complete graph on the set $\{u_1,u_2,\ldots,u_n,v_1,v_2,\ldots,v_n\}$ and $\Gamma$ be the graph obtained from $K_{2n}$ by removing the edges jointing $u_i$ and $v_i$, $i=1,2,\ldots,n$.
Let $\alpha=(u_1\,u_2\,\ldots\,u_n)$ and $\beta=(v_1\,v_2\,\ldots\,v_n)$ be two permutations on $V(\Gamma)$. It is a simple matter to check that $(\alpha,\beta)$ is a TFS-morphism of $\Gamma$. By Theorem \ref{others} (iv), $\Gamma\ltimes\Sigma$ is stable. It is obvious that $\Gamma$ is connected and non-bipartite. Furthermore, it is straightforward to check that $\Gamma[a]\ne\Gamma[b]$ for each pair of distinct  vertices $a$ and $b$ of $\Gamma$. Since $\Sigma$ is connected, non-bipartite and $R$-thin, by Lemma \ref{cnbthss} we have that $\Gamma\ltimes\Sigma$ is connected, non-bipartite and $R$-thin. Therefore $\Gamma\ltimes\Sigma$ is nontrivially unstable.
\end{proof}
Now we use graph products to construct nontrivially unstable circulant graphs.
\begin{exam}
\label{cpc1}
Let $\Gamma=\Cay(\mathbb{Z}_{30},\{\pm1,\pm4\})$ and $\Sigma=\Cay(\mathbb{Z}_n,\{1,-1\})$ where $n\geq3$. For each $i\in \mathbb{Z}_n$, define a mapping $t_i:\mathbb{Z}_n\rightarrow\mathbb{Z}_n,~x\mapsto x+i,\forall x\in \mathbb{Z}_n$.  Let $\Lambda=\Gamma\times^{p}\Sigma$ where $p$ is the mapping from $V(\Gamma)\times V(\Sigma)$ to $\Aut(\Sigma)$ defined as follows:
\begin{equation*}
p(a,b)=\left\{
          \begin{array}{ll}
            t_1, & \hbox{if}~b-a=1~\mbox{or}~-4;\\
            t_{-1}, & \hbox{if}~b-a=-1~\mbox{or}~4;\\
            t_0, & \hbox{otherwise.}
          \end{array}
        \right.
\end{equation*}
Then $\Lambda$ is an unstable Cayley graph on $\mathbb{Z}_{30}\times\mathbb{Z}_n$. In particular, if $(30,n)=1$, then $\Lambda$ is a nontrivially unstable circulant graph.
\end{exam}
\begin{proof}
Clearly, $t_0$ is the identity element of $\Aut(\Sigma)$ and $t_1,t_{-1}$ are two mutually inverse elements of $\Aut(\Sigma)$. Therefore $p$ is well defined. Define two permutations $\alpha$ and $\beta$ on $\mathbb{Z}_{30}$ by the rule: $a^\alpha=11a$ and $a^\beta=11a+15$ for all $a\in \mathbb{Z}_{30}$. Then $(\alpha,\beta)$ is a nontrivial TF-morphism of $\Gamma$ and $p(a^{\alpha},b^{\beta})=p(11a,11b+15)=p(a,b)$. By Theorem \ref{dbundle} (i), $\Lambda$ is unstable. It is a simple matter to verify that
\begin{equation*}
\Lambda=\Cay(\mathbb{Z}_{30}\times\mathbb{Z}_n,
\{\pm(1,0),\pm(1,2),\pm(4,0),\pm(4,-2)\})
\end{equation*}
Moreover, if $n$ is odd, then $\Lambda$ is connected, $R$-thin, and non-bipartite.

If $(30,n)=1$, then $\mathbb{Z}_{30}\times\mathbb{Z}_n\cong \mathbb{Z}_{30n}$ and $\Lambda$ is isomorphic to
\begin{equation*}
\Cay(\mathbb{Z}_{30n},
\{\pm n,\pm(n+60),\pm 4n,\pm(4n-60)\}),
\end{equation*}
which is a nontrivially unstable circulant graph.
\end{proof}
\begin{exam}
\label{cpc2}
Let $\Gamma=\Cay(\mathbb{Z}_{n},\{\pm1\})$ and $\Sigma=\Cay(\mathbb{Z}_{2m},\mathbb{Z}_{2m}\setminus\{1\})$ with $\gcd(n,2m)=1$. Let $\Lambda=\Gamma\times^{p}\Sigma$ where $p$ is the mapping from $V(\Gamma)\times V(\Gamma)$ to $\Aut(\Sigma)$ such that $x^{p(a,b)}=x+m$
for every $x\in \mathbb{Z}_{2m}$ and $(a,b)\in V(\Gamma)\times V(\Sigma)$. Then $\Lambda$ is a nontrivially unstable circulant graph.
\end{exam}
\begin{proof}
Since  $\Sigma=\Cay(\mathbb{Z}_{2m},\mathbb{Z}_{2m}\setminus\{1\})\cong K_{2m}$, we have $\Aut(\Sigma)=\Sym(\mathbb{Z}_{2m})$. Let $\rho=(1,2,\cdots,2m-1),\,\delta=(0,m)(1,1+m)\cdots(m-1,2m-1)\in \Sym(\mathbb{Z}_{2m})$. Then $\delta$ is an involution, $\delta^{-1}\rho\delta\neq \rho$ and $p(a,b)=\delta$ for all $(a,b)\in V(\Gamma)\times V(\Gamma)$. By Proposition \ref{dihedral}, $\Lambda$ is unstable.
It is a simple matter to verify that $\Lambda$ is connected, $R$-thin, nonbipartite and a Cayley graph on $\mathbb{Z}_{n}\times\mathbb{Z}_{2m}$ $(\cong\mathbb{Z}_{2mn})$. Therefore $\Lambda$ is a nontrivially unstable circulant graph.
\end{proof}
\begin{exam}
\label{cpc3}
Let $\Gamma=\Cay(\mathbb{Z}_{n},\{\pm1\})$ and $\Sigma=\Cay(\mathbb{Z}_{12},\{\pm1,\pm2,\pm7\})$ where $n$ is a positive integer coprime to $12$. Let $\Lambda=\Gamma\times^{p}\Sigma$ where $p$ is the mapping from $V(\Gamma)\times V(\Gamma)$ to $\Aut(\Sigma)$ such that $x^{p(a,b)}=x+6$
for every $x\in \mathbb{Z}_{12}$ and $(a,b)\in V(\Gamma)\times V(\Sigma)$. Then $\Lambda$ is a nontrivially unstable circulant graph.
\end{exam}
\begin{proof}
Note that  $\Sigma\cong \Cay(\mathbb{Z}_{3},\{\pm1\}) \times\Cay(\mathbb{Z}_{4},\{\pm1,2\})$ and the permutation
\begin{equation*}
\rho:=(1,7,10)(2,5,11)(3,6,9)
\end{equation*}
is an automorphism of $\Sigma$.  Let $\delta=(0,6)(1,7)(2,8)(3,9)(4,10)(5,11)$. Then $\delta$ is an involution, $\delta^{-1}\rho\delta\neq \rho$ and $p(a,b)=\delta$ for all $(a,b)\in V(\Gamma)\times V(\Gamma)$. By Proposition \ref{dihedral}, $\Lambda$ is unstable.
It is a simple matter to verify that
$\Lambda$ is  connected, $R$-thin, nonbipartite and isomorphic to the following Cayley graph
\begin{equation*}
 \Cay(\mathbb{Z}_{12n}, \{\pm(n+12),\pm(4n+12),\pm(5n+12),
 \pm(7n+12),\pm(8n+12),\pm(11n+12),\}).
\end{equation*}
Therefore $\Lambda$ is a nontrivially unstable circulant graph.
\end{proof}
\begin{exam}
\label{strpex}
Let $\Gamma=\Cay(\mathbb{Z}_{2n},\{\pm1\})$ and $\Sigma=\Cay(\mathbb{Z}_{9},\{\pm1,\pm4,\pm7\})$ where $3\nmid n$.
Then both $\Gamma\Box\Sigma$ and $\Gamma\boxtimes\Sigma$ are nontrivially unstable circulant graphs.
\end{exam}
\begin{proof}
Clearly, $\Gamma$ is bipartite. Note that $\Sigma$ is $R$-thick as $0$ and $4$ have the same neighbours. Of course, $\Sigma$ has a nontrivial TF-morphism. By Theorem \ref{others} (i) and (ii), $\Gamma\Box\Sigma$ and $\Gamma\boxtimes\Sigma$ are both unstable. Using Lemma \ref{cnthcp} and \ref{cnthsp}, it is straightforward to check that $\Gamma\Box\Sigma$ and $\Gamma\boxtimes\Sigma$ are both connected, non-bipartite and $R$-thin. Note that $\Gamma\Box\Sigma$ and $\Gamma\boxtimes\Sigma$ are isomorphic to
\begin{equation*}
\Cay(\mathbb{Z}_{18n},\{\pm9,\pm2n,\pm8n,\pm14n\}).
\end{equation*}
and
\begin{equation*}
\Cay(\mathbb{Z}_{18n},\{\pm9,\pm(2n\pm9),
\pm(8n\pm9),\pm(14n\pm9),
\pm2n,\pm8n,\pm14n\})
\end{equation*}
respectively.
Therefore both $\Gamma\Box\Sigma$ and $\Gamma\boxtimes\Sigma$ are nontrivially unstable circulant graphs.
\end{proof}
\begin{exam}
\label{strex}
Let $\Gamma=\Cay(\mathbb{Z}_{10},\{\pm3,\pm4,5\})$ and $\Sigma=\Cay(\mathbb{Z}_{n},\{1,-1\})$ where $n$ is coprime to $10$.
Then $\overline{\Gamma\boxtimes\Sigma}$ is a nontrivially unstable circulant graph.
\end{exam}
\begin{proof}
Note that $\overline{\Gamma}=\Cay(\mathbb{Z}_{10},\{\pm1,\pm2\})$ which has a TF-morphism $(\alpha,\beta)$ with $x^{\alpha}=3x$ and $x^{\beta}=3x+5$ for every $x\in \mathbb{Z}_{10}$. Since $x^{\alpha\beta^{-1}}=7(3x)+5=x+5$, we have that $\alpha\beta^{-1}$ is a derangement. By Theorem \ref{others} (iii), $\overline{\Gamma\boxtimes\Sigma}$ is unstable. It is a simple matter to verify that $\overline{\Gamma\boxtimes\Sigma}$ is a connected, non-bipartite and $R$-thin Cayley graph on $\mathbb{Z}_{10n}$. Therefore $\overline{\Gamma\boxtimes\Sigma}$ is a nontrivially unstable circulant graph.
\end{proof}
\begin{exam}
\label{semiex1}
Let $\Gamma=\Cay(\mathbb{Z}_{n},\{1,-1\})$ and $\Sigma=\Cay(\mathbb{Z}_{10},\{\pm1,\pm2\})$ where $n$ is coprime to $10$.
Then $\Gamma\ltimes\Sigma$ is a nontrivially unstable circulant graph.
\end{exam}
\begin{proof}
Let $\alpha$ and $\beta$ be the two permutations on $V(\Sigma)$ defined by the rule $x^{\alpha}=3x$ and $x^{\beta}=3x+5$. It is straightforward to check that $(\alpha,\beta)$ is a TF-morphism of $\Sigma$. Therefore $\Sigma$ is unstable.  By Theorem \ref{others} (iv), $\Gamma\ltimes\Sigma$ is unstable. Obviously, $\Gamma$ is connected and $\Sigma$ is connected, non-bipartite and $R$-thin. By Lemma \ref{cnbthss}, $\Gamma\ltimes\Sigma$ is connected, non-bipartite and $R$-thin. Note that $\Gamma\ltimes\Sigma$ is isomorphic to
\begin{equation*}
\Cay(\mathbb{Z}_{10n},\{\pm(n+10),\pm(n-10),\pm(2n+10),\pm(2n-10),\pm n,\pm 2n\}).
\end{equation*}
Therefore $\Gamma\ltimes\Sigma$ is a nontrivially unstable circulant graph.
\end{proof}

\begin{exam}
\label{semiex2}
Let $\Gamma=\Cay(\mathbb{Z}_{10},\{\pm3,\pm4,5\})$ and $\Sigma=\Cay(\mathbb{Z}_{n},\{1,-1\})$ where $n$ is coprime to $10$. Then$\Gamma\ltimes\Sigma$ is a nontrivially unstable circulant graph.
\end{exam}
\begin{proof}
As shown in the proof of Example \ref{strex}, $\Gamma$ has a TFS-morphism. By Theorem \ref{others} (iv), $\Gamma\ltimes\Sigma$ is unstable. Obviously, $\Gamma\ltimes\Sigma$ is isomorphic to
\begin{equation*}
\Cay(\mathbb{Z}_{10n},\{\pm(3n+10),\pm(3n-10),\pm(4n+10),\pm(4n-10),\pm (5n+10),\pm 10\})
\end{equation*}
 which is a connected, non-bipartite and $R$-thin. Therefore $\Gamma\ltimes\Sigma$ is a nontrivially unstable circulant graph.
\end{proof}
\begin{exam}
\label{semiex3}
Let $\Lambda=\Gamma\ltimes\Sigma$ where $\Gamma=\Cay(\mathbb{Z}_{8},\{\pm1,\pm2,\pm3\})$ and $\Sigma=\Cay(\mathbb{Z}_{n},\{\pm1\})$ with $2\nmid n$. Then $\Lambda$ is a nontrivially unstable circulant graph.
\end{exam}
\begin{proof}
Note that $\Gamma$ is a graph obtained from $K_8$ by removing a perfect
matching. Since $\Sigma$ is an odd cycle which is connected and non-bipartite,  by Proposition \ref{k2n}, we have that $\Lambda$ is nontrivially unstable.  Obviously,
\begin{align*}
  \Lambda=& \Cay(\mathbb{Z}_{8}\times\mathbb{Z}_{n},
\{\pm(1,1),\pm(1,-1),\pm(2,1),\pm(2,-1),\pm(3,1),\pm(3,-1),\pm(0,1)\}) \\
  \cong& \Cay(\mathbb{Z}_{8n},S)
\end{align*}
where $S=\{\pm(n+8),\pm(n-8),\pm(2n+8),\pm(2n-8),\pm(3n+8),\pm(3n-8),\pm 8\}$.
Therefore $\Lambda$ is a nontrivially unstable circulant graph.
\end{proof}

\begin{exam}
\label{lexiex}
Let $\Gamma=\Cay(\mathbb{Z}_{n},\{\pm1\})$ and $\Sigma=\Cay(\mathbb{Z}_{2m},\{\pm1\})$ where $n$ is coprime to $2m$.
Then $\Gamma[\Sigma]$ is a nontrivially unstable circulant graph.
\end{exam}
\begin{proof}
Obviously, $\Gamma$ is connected and non-bipartite and $\Sigma$ is connected and $R$-thin. By Lemma \ref{lpthin}, $\Gamma[\Sigma]$ is connected, non-bipartite and $R$-thin. Note that $\Gamma[\Sigma]$ is isomorphic to $\Cay(\mathbb{Z}_{n}\times\mathbb{Z}_{2m},
\{\pm(1,x)\mid x\in \mathbb{Z}_{2m}\}\cup\{\pm(0,1)\})$ which is a circulant graph. Since $\Sigma$ is a cycle of even length, it has a nontrivial TF-morphism. By Theorem \ref{others} (v), $\Gamma[\Sigma]$ is a nontrivially unstable circulant graph.
\end{proof}
\section{On the stability of circulant graphs}
\label{sec7}

As mentioned in the introduction, Wilson \cite{W2008} proposed four sufficient conditions for the instability of circulant graphs. However, two of these conditions were later identified as flawed and corrected by Qin et al. \cite{QXZ2019} and Hujdurovi\'c et al. \cite{HMM2021}, respectively. The following lemma incorporates Wilson's two original valid conditions alongside the corrected versions proposed by Qin et al. and Hujdurovi\'c et al..
\begin{lem}
\label{wtype}
Let $X = \Cay(\mathbb{Z}_n, S)$ be a circulant graph of even order.
Let $S_e = S \cap2\mathbb{Z}_n$ and $S_o = S \setminus S_e$. If any of the following conditions is true, then $X$ is unstable.
\begin{enumerate}
  \item[$\mathrm{C.}1.\,$] There is a nonzero element $h$ of $2\mathbb{Z}_n$, such that $h + S_e = S_e$.
  \item[$\mathrm{C.}2'.$] $n$ is divisible by $4$, and there exists $h \in 1 + 2\mathbb{Z}_n$, such that
\begin{enumerate}
  \item[{\rm(a)}] $2h + S_o = S_o$, and
  \item[{\rm(b)}]for each $s \in S$ with $s\equiv0~\mbox{or}~-h\pmod4$, we have $s + h\in S$.
\end{enumerate}
  \item[$\mathrm{C.}3'.$] There is a subgroup $H$ of $\mathbb{Z}_n$, such that the set
\begin{equation*}
R = \{s\in S \mid s +H \not\subseteq S \},
\end{equation*}
is nonempty and has the property that if we let $d =\gcd(R\cup\{n\})$, then $n/d$ is even,
$r/d$ is odd for every $r \in R$, and either $H\not\subseteq d\mathbb{Z}_n$ or $H\subseteq 2d\mathbb{Z}_n$.
  \item[$\mathrm{C.}4.\,$] There exists $r \in \mathbb{Z}_{n}^{*}$, such that $(n/2)+rS=S$.
\end{enumerate}
\end{lem}
\begin{remark}
The statements $\mathrm{C}.2'$ and $\mathrm{C}.3'$, due to Qin et al. \cite{QXZ2019} and Hujdurovi\'c et al. \cite{HMM2021} respectively, are slightly corrected versions of the original statements of
\cite[Theorems C.2 and C.3]{W2008}.
\end{remark}
The first example of a nontrivially unstable circulant graph that does not satisfy any conditions in Lemma \ref{wtype} was introduced by Qin et al. \cite{QXZ2019}. Subsequently, Hujdurovi\'c et al. \cite{HMM2021} established new sufficient conditions for the instability of circulant graphs. Their results demonstrate that these conditions can generate infinitely many nontrivially unstable circulant graphs, even when the criteria in Lemma \ref{wtype} are not met. The work of Hujdurovi\'c et al. is formalized in the following three lemmas.

\begin{lem}[{\cite[Theorem 3.2]{HMM2021}}]
\label{hmmtype1}
Let $\Gamma=\Cay(\mathbb{Z}_{2m}, S)$ be a circulant graph. Choose nontrivial subgroups $H$ and $K$ of $\mathbb{Z}_{2m}$, such that $|K|$ is even and let $K_o=K\setminus 2K$.
If either
\begin{enumerate}
 \item[\rm(1)] $S+H\subseteq S\cup(K_o+H)$ and $H\cap K_o =\emptyset$ or
 \item[\rm(2)] $(S\setminus K_o)+H\subseteq S\cup K_o$ and either $|H|\neq2$ or $4\mid|K|$,
\end{enumerate}
then $\Gamma$ is not stable.
\end{lem}
\begin{lem}[{\cite[Proposition 3.7]{HMM2021}}]
\label{hmmtype2}
Let $\Gamma=\Cay(\mathbb{Z}_{2m}, S)$ be a circulant graph.
If $\Gamma\cong\Cay(\mathbb{Z}_{2m}, S+m)$, then $\Gamma$ is unstable.
\end{lem}
\begin{lem}[{\cite[Proposition 3.12]{HMM2021}}]
\label{hmmtype3}
Let $\Gamma=\Cay(\mathbb{Z}_{2m}, S)$ be a circulant graph.
If there is a nontrivial TF-morphism $(\alpha,\beta)$ of $\Cay(2\mathbb{Z}_{2m}, (2\mathbb{Z}_{2m})\cap S)$ and a subgroup $H$ of $\mathbb{Z}_{2m}$ such that $v+H\subseteq S$ for all $v\in S\setminus2\mathbb{Z}_{2m}$ and $v^\alpha-v,v^\beta-v\in H$ for all $v\in2\mathbb{Z}_{2m}$, then $\Gamma$ is unstable.
\end{lem}
\begin{remark}
Lemma \ref{hmmtype1} extends Wilson's conditions $\mathrm{C.}1$, $\mathrm{C.}2'$ and $\mathrm{C.}3'$ (see \cite[Proposition 3.4]{HMM2021}), while Lemma \ref{hmmtype2} generalizes Wilson's condition $\mathrm{C.}4$. On the other hand, Lemma \ref{hmmtype3} provides a method for constructing large unstable circulant graph from smaller ones.
\end{remark}
\medskip
Inspired by Proposition \ref{dihedral}, we investigate the role of involutory automorphisms in graph instability and derive a new sufficient condition for the instability of circulant graphs (Theorem \ref{ncon}). We now restate and prove Theorem \ref{ncon}.
\begin{ncon}
Let $\Gamma=\Cay(\mathbb{Z}_{2m},S)$ be a circulant graph. If $\Cay\big(\mathbb{Z}_{2m},(S\setminus\{m\})+m\big)$ has an automorphism fixing $0$ but moving $m$, then $\Gamma$ is unstable.
\end{ncon}

\begin{proof}
Let $\sigma$ be an automorphism of $\Cay\big(\mathbb{Z}_{2m},(S\setminus\{m\})+m\big)$ such that $0^\sigma=0$ and  $m^\sigma\neq m$. Then $y-x\in (S\setminus\{m\})+m\Longleftrightarrow y^\sigma-x^\sigma\in (S\setminus\{m\})+m$ for each pair of elements $x,y\in \mathbb{Z}_{2m}$. Define two permutations $\alpha$ and $\beta$ on $\mathbb{Z}_{2m}$ by the rule:
$x^\alpha=x^\sigma+m$ and $x^\beta=(x+m)^\sigma$ for all $x\in\mathbb{Z}_{2m}$. Then $\alpha\neq \beta$ as $0^\alpha=0^\sigma+m=m$ and $0^\beta=m^\sigma\neq m$.
Since
\begin{equation*}
  (x+m)^{\beta}-x^{\alpha}=x^{\sigma}-(x^{\sigma}+m)=m
\end{equation*}
and
\begin{align*}
  x\sim_{\Gamma}y\Longleftrightarrow&
  y-x\in S\setminus\{m\}\\
\Longleftrightarrow&y+m-x\in (S\setminus\{m\})+m\\
\Longleftrightarrow&(y+m)^{\sigma}-x^{\sigma}\in (S\setminus\{m\})+m\\
\Longleftrightarrow&(y+m)^{\sigma}-(x^{\sigma}+m)\in (S\setminus\{m\})\\
\Longleftrightarrow&y^\beta-x^\alpha\in (S\setminus\{m\})\\
\Longleftrightarrow&y^\beta\sim_{\Gamma}x^\alpha
\end{align*}
for each pair of elements $x,y\in \mathbb{Z}_{2m}$ with $y\neq x+m$, it follows that $(\alpha,\beta)$ is a nontrivial TF-morphism of $\Gamma$. By Lemma \ref{TF}, $\Gamma$ is unstable.
\end{proof}
The following lemma presents a method for constructing a circular graph that satisfies the condition in Theorem \ref{ncon}.
\begin{lem}
\label{oldtonew}
Let $h$ be a positive integer and let $S$ be an inverse-closed subset of $\mathbb{Z}_{4h}$. Let $K$ be a subgroup of $\mathbb{Z}_{4h}$ of order twice an odd integer and set $K_o=K\setminus 2K$. If $(S\setminus K_o)+\langle h\rangle\subseteq S\cup K_o$, then the Cayley graph $\Cay\big(\mathbb{Z}_{4h},(S\setminus\{2h\})+2h\big)$ has an automorphism fixing $0$ but moving $2h$.
\end{lem}
\begin{proof}
 Set $\Sigma=\Cay\big(\mathbb{Z}_{4h},S^*)$ where $S^*=(S\setminus\{2h\})+2h$. Define a permutation $\sigma$ on $\mathbb{Z}_{4h}$ as follows:
\begin{equation*}
  x^\sigma=\left\{
             \begin{array}{ll}
               x, & \hbox{if}~x\in 2K; \\
               x+2h, & \hbox{if}~x\in h+2K; \\
               x-h, & \hbox{otherwise.}
             \end{array}
           \right.
\end{equation*}
It is straightforward to verify that $\sigma$ is a well defined permutation fixing $0$ but moving $2h$. We now show that $\sigma$ is an automorphism of $\Sigma$.
Let $x$ and $y$ be two adjacent elements in $\Sigma$. Then $x-y\in S^*$ and hence $x-y+2h\in S\setminus\{2h\}$. To prove that $\sigma$ is an automorphism, we verify that $x^{\sigma}-y^{\sigma}\in S^*$ by examining the following three cases.

\textsf{Case 1:}  $x,y\notin \langle h,K\rangle$ or $x$ and $y$ lie in the same left coset of $2K$ in $\mathbb{Z}_{4h}$.

Here, $x^{\sigma}-y^{\sigma}=x-y$, so $x^{\sigma}-y^{\sigma}\in S^*$ holds trivially.

\textsf{Case 2:}  Exactly one of $x$ or $y$ is in $\langle h,K\rangle$.

In this case, $(x-y+\langle h\rangle)\cap K_o=\emptyset$. Since $x-y+2h\in S\setminus\{2h\}$ and $(S\setminus K_o)+\langle h\rangle\subseteq S\cup K_o$, it follows that $x-y+\langle h\rangle\subseteq S\setminus K_o$. Observing $x^{\sigma}-y^{\sigma}\in x-y+\langle h\rangle$, we have $x^{\sigma}-y^{\sigma}\in S\setminus K_o$. Consequently, $x^{\sigma}-y^{\sigma}+2h\in S\setminus\{2h\}$ and hence $x^{\sigma}-y^{\sigma}\in S^*$.

\textsf{Case 3:}  $x$ and $y$  belong to different  left cosets of $2K$ in  $\langle h,K\rangle$.

We analyze the possible subcases:
\begin{itemize}
  \item If $x\in 2h+2K$ and $y\in 2K$, then $x-y+2h\in 2K$ and $x^{\sigma}-y^{\sigma}+2h=x-y+h\notin K_o$.
  \item If $x\in h+2K$ and $y\in 2K$, then $x-y+2h\in -h+2K$ and  $x^{\sigma}-y^{\sigma}+2h=x-y\notin K_o$.
  \item If $x\in -h+2K$ and $y\in 2K$, then $x-y+2h\in h+2K$ and  $x^{\sigma}-y^{\sigma}+2h=x-y+h\notin K_o$.
    \item If $x\in 2K$ and $y\in 2h+2K$, then $x-y+2h\in 2K$ and  $x^{\sigma}-y^{\sigma}+2h=x-y-h\notin K_o$.
  \item If $x\in h+2K$ and $y\in 2h+2K$, then $x-y+2h\in h+2K$ and  $x^{\sigma}-y^{\sigma}+2h=x-y+h\notin K_o$.
  \item If $x\in -h+2K$ and $y\in 2h+2K$, then $x-y+2h\in -h+2K$ and  $x^{\sigma}-y^{\sigma}+2h=x-y+2h\notin K_o$.
  \item If $x\in 2h+2K$ and $y\in h+2K$, then $x-y+2h\in -h+2K$ and  $x^{\sigma}-y^{\sigma}+2h=x-y-h\notin K_o$.
  \item If $x\in 2K$ and $y\in h+2K$, then $x-y+2h\in h+2K$ and $x^{\sigma}-y^{\sigma}+2h=x-y\notin K_o$.
  \item If $x\in -h+2K$ and $y\in h+2K$, then $x-y+2h\in 2K$ and $x^{\sigma}-y^{\sigma}+2h=x-y-h\notin K_o$.
    \item If $x\in 2h+2K$ and $y\in -h+2K$, then $x-y+2h\in h+2K$ and $x^{\sigma}-y^{\sigma}+2h=x-y+2h\notin K_o$.
  \item If $x\in h+2K$ and $y\in -h+2K$, then $x-y+2h\in 2K$ and $x^{\sigma}-y^{\sigma}+2h=x-y+h\notin K_o$.
  \item If $x\in 2K$ and $y\in -h+2K$, then $x-y+2h\in -h+2K$ and $x^{\sigma}-y^{\sigma}+2h=x-y-h\notin K_o$.
\end{itemize}
In all subcases, $x-y+2h\notin K_o$ and $x^{\sigma}-y^{\sigma}+2h\in (x-y+2h+\langle h\rangle)\setminus K_o$. Since $x-y+2h\in S\setminus\{2h\}$ and $(S\setminus K_o)+\langle h\rangle\subseteq S\cup K_o$, it follows that  $x^{\sigma}-y^{\sigma}+2h\in S\setminus \{2h\}$. Therefore, $x^{\sigma}-y^{\sigma}\in S^*$, completing the proof.
\end{proof}

\medskip
We have been unable to find a method for constructing a circular graph that satisfies the condition in Theorem \ref{ncon} but fails to satisfy the condition in Lemma \ref{oldtonew}. This leads us to propose the following question.
\begin{ques}
Let $h$ be a positive integer, and let $S$ be an inverse-closed subset of $\mathbb{Z}_{4h}$ such that the circulant graph $\Cay\big(\mathbb{Z}_{4h},(S\setminus\{2h\})+2h\big)$ has an automorphism fixing $0$ but moving $2h$. Is there a subgroup $K$ of $\mathbb{Z}_{4h}$ of order twice an odd integer such that $(S\setminus K_o)+\langle h\rangle\subseteq S\cup K_o$ where $K_o=K\setminus 2K$?
\end{ques}

\medskip
We observe that multiple conditions in Lemmas \ref{hmmtype1},~\ref{hmmtype2} and Theorem \ref{ncon} may apply to a single unstable circulant graph. To refine these overlapping conditions, we establish the following lemma.
\begin{lem}
\label{equi}
Let $\Gamma=\Cay(\mathbb{Z}_{2m},S)$ be a circulant graph which satisfies the conditions in Lemma \ref{hmmtype1}.
\begin{enumerate}
  \item If $2\nmid|H|$, then $(S\setminus L_o)+H=S\setminus L_o$ where $L=K+H$ and $L_o=L\setminus2L$.
  \item If $2\parallel|H|$ and $m\notin H\setminus K_{o}$, then $|H|>2$ and $(S\setminus L_o)+2H=S\setminus L_o$ where $L=K+H$ and $L_o=L\setminus2L$.
  \item If $m\in H\setminus K_{o}$, then $\Gamma\cong\Cay(\mathbb{Z}_{2m}, S+m)$.

  \item If $4\mid|H|$ and $m\notin H\setminus K_{o}$, then $\Cay\big(\mathbb{Z}_{2m},(S\setminus\{m\})+m\big)$ has an automorphism fixing $0$ but moving $m$.
\end{enumerate}
\end{lem}
\begin{proof}
(i) Set $L=K+H$ and $L_o=L\setminus 2L$. Since $2\nmid|H|$ and $2\mid|K|$, it follows that $2\mid|L|$ and $L_o=K_{o}+H$.
Consequently, either the inclusion $S+H\subseteq S\cup(K_o+H)$ or the inclusion $(S\setminus K_o)+H\subseteq S\cup K_o$ implies that $(S\setminus L_o)+H=S\setminus L_o$.

\medskip
(ii)~Since $2\parallel|H|$, $2\mid |K|$ and $m\notin H\setminus K_{o}$, it follows that $m\in H\cap K_o$ and $2\parallel|K|$. Consequently, conditions (1) and (2) in Lemma \ref{hmmtype1} imply $(S\setminus K_o)+H\subseteq S\cup K_o$ and $|H|>1$. Set $L=K+H$ and $L_o=L\setminus 2L$. Then $L_o=K_o+2H$ and we have
 \begin{equation*}
(S\setminus L_o)+2H\subseteq (S\setminus K_o)+H\subseteq S\cup K_o\subseteq S\cup L_o.
\end{equation*}
Observing that $\left((S\setminus L_o)+2H\right)\cap L_o=\emptyset$, the inclusion $(S\setminus L_o)+2H\subseteq S\cup L_o$ implies $(S\setminus L_o)+2H=S\setminus L_o$.

\medskip
(iii)~Since $m\in H\setminus K_{o}$ and $2\mid|K|$, we have $2\mid |H|$ and $4\mid |K|$. First, assume condition (1) of Lemma \ref{hmmtype1} holds. Let $L=K+H$ and $L_o=L\setminus 2L$. Since $H\cap K_o=\emptyset$, we conclude that the Sylow $2$-subgroup of $H$ is a proper subgroup of the Sylow 2-subgroup of $K$. Consequently, $L_o=L_o+H=K_o+H$. This together with the containment $S+H\subseteq S\cup(K_o+H)$ implies $(S\setminus L_o)+H=S\setminus L_o$. Since $m\in H$, it follows that $(S\setminus L_o)+m=S\setminus L_o$. Next, assume condition (2) of Lemma \ref{hmmtype1} holds, and set $L=K$. Then $(S\setminus L_o)+H\subseteq S\cup L_o$, which implies $(S\setminus L_o)+m\subseteq S\cup L_o$. Observing $m\in L\setminus L_o$, we get $L_o+m=L_o$. Consequently, $\left((S\setminus L_o)+m\right)\cap L_o=\emptyset$ and therefore $(S\setminus L_o)+m=S\setminus L_o$. Now we have proved that the equation $(S\setminus L_o)+m=S\setminus L_o$ always holds.

Let $\ell$ be the index of $L$ in $\mathbb{Z}_{2m}$. Then $\{0,1,\ldots,\ell-1\}$ forms a left transversal of $L$ in $\mathbb{Z}_{2m}$. Consequently, every element of
$\mathbb{Z}_{2m}$ can be uniquely expressed as $i+x$, where $i\in\{0,1,\ldots,\ell-1\}$ and $x\in L$. Define a permutation $\sigma$ on $\mathbb{Z}_{2m}$ by the rule:
\begin{equation*}
  (i+x)^{\sigma}=\left\{
              \begin{array}{ll}
       i+x+m, & \hbox{if}~x\in L_o; \\
    i+x, & \hbox{otherwise}.
      \end{array}
            \right.
\end{equation*}
Consider two adjacent vertices $i+x$ and $j+y$ in the graph $\Gamma$. We have
\begin{equation*}
(i-j)+(x-y)=(i+x)-(j+y)\in S.
\end{equation*}
By the definition of $\sigma$, we get
\begin{equation*}
(j+y)^{\sigma}-(i+x)^{\sigma}
=\left\{
   \begin{array}{ll}
     (i-j)+(x-y), & \hbox{if either}~x,y\in L_{o}~\hbox{or}~x,y\notin L_{o};\\
     (i-j)+(x-y)+m, & \hbox{otherwise.}
   \end{array}
 \right.
\end{equation*}
Recalling $(S\setminus L_o)=(S\setminus L_o)+m$, it follows that $(j+y)^{\sigma}-(i+x)^{\sigma}\in (S\setminus L_o)+m$ whenever $(i-j)+(x-y)\in S\setminus L_o$. Now assume  $(i-j)+(x-y)\in S\cap L_o$. Then $i=j$ and exactly one of $x$ and $y$ is contained in $L_o$. Consequently,
\begin{equation*}
(i+x)^{\sigma}-(j+y)^{\sigma}=x-y+m\in (S\cap L_o)+m.
\end{equation*}
Thus we always have $(i+x)^{\sigma}-(j+y)^{\sigma}\in S+m$, which means that $\sigma$ is an isomorphism from $\Gamma$ to $\Cay(\mathbb{Z}_{2m}, S+m)$.

\medskip
(iv) Since $4\mid|H|$ and $m\notin H\setminus K_{o}$, there exists $h\in H$ such that $m=2h\in K_o$. Consequently,  $2\parallel|K|$, and conditions (1) and (2) in Lemma \ref{hmmtype1} imply that $(S\setminus K_o)+\langle h\rangle\subseteq S\cup K_o$. By Lemma \ref{oldtonew}, the Cayley graph $\Cay\big(\mathbb{Z}_{2m},(S\setminus\{2h\})+m\big)$ has an automorphism fixing $0$ but moving $2h$.
\end{proof}

By applying Lemma \ref{equi}, we consolidate Lemmas \ref{hmmtype1},~\ref{hmmtype2} and Theorem \ref{ncon} into Theorem \ref{hmmtype}, restated below.
\begin{hmmtype}
Let $\Gamma=\Cay(\mathbb{Z}_{2m}, S)$ be a circulant graph.
Then $\Gamma$ is unstable if any of the following conditions is true:
\begin{enumerate}
  \item $(S\setminus K_o)+H=S\setminus K_o$ where $K$ is an even order subgroup of $\mathbb{Z}_{2m}$, $K_o=K\setminus 2K$ and $H$ is an odd order subgroup of $K$;
  \item $\Gamma\cong\Cay(\mathbb{Z}_{2m}, S+m)$;
  \item $\Cay\big(\mathbb{Z}_{2m},(S\setminus\{m\})+m\big)$ has an automorphism fixing $0$ but moving $m$.
\end{enumerate}
\end{hmmtype}

\section{Concluding remark}
\label{sec8}

Recently, Jiang and Zhang proposed a method for constructing unstable graphs from bipartite graphs (\cite[Construction 3.1]{JZ2025}). They applied this construction to Cayley graphs, deriving a sufficient condition for the instability of Cayley graphs of groups containing a subgroup of index $2$ (\cite[Theorem 4.1]{JZ2025}). The following proposition focuses on the case of circulant graphs, refining \cite[Theorem 4.1]{JZ2025} for this specific class.
\begin{pro}
\label{preserving}
Let $\Gamma=\Cay(\mathbb{Z}_{2m}, S)$ be a circulant graph. Let $S_o=S\setminus 2\mathbb{Z}_{2m}$ and $S_e=S\cap 2\mathbb{Z}_{2m}$. Suppose the circulant graph $\Gamma_{o}:=\Cay(\mathbb{Z}_{2m}, S_o)$ admits two distinct automorphisms, $\sigma$ and $\rho$, with the following properties:
\begin{enumerate}
  \item both $\sigma$ and $\rho$ preserve $2\mathbb{Z}_{2m}$ set-wise;
  \item $y^\rho-x^\sigma\in S_e$ for any $x,y\in \mathbb{Z}_{2m}$ such that $y-x\in S_e$.
\end{enumerate}
Then $\Gamma$ is unstable.
\end{pro}
\begin{remark}
\label{rem}
Let $\Gamma=\Cay(\mathbb{Z}_{2m}, S)$ be a circulant graph satisfying the condition in Lemma \ref{preserving}. Then the Cayley graph $\Gamma_{e}:=\Cay(\mathbb{Z}_{2m}, S_e)$ decomposes into a disjoint union of $\Gamma'_{e}$ and $\Gamma''_{e}$, where $\Gamma'_{e}$ and $\Gamma''_{e}$ are the
subgraphs of $\Gamma_{e}$ induced by
$2\mathbb{Z}_{2m}$ and $1+2\mathbb{Z}_{2m}$, respectively. Clearly, $\Gamma'_{e}:=\Cay(2\mathbb{Z}_{2m}, S_e)$ and $\Gamma''_{e}\cong\Gamma'_{e}$. Since both $\sigma$ and $\rho$ preserve $2\mathbb{Z}_{2m}$, they also preserve $1+2\mathbb{Z}_{2m}$. Condition (ii) in Lemma \ref{preserving} ensures that the restriction of $(\sigma,\rho)$ to $2\mathbb{Z}_{2m}$ is a TF-morphism of $\Gamma'_{e}$, and its restriction to $1+2\mathbb{Z}_{2m}$ is a TF-morphism of $\Gamma''_{e}$. Since $\sigma\ne\rho$, we conclude that the circulant graph $\Gamma'_{e}$ has a nontrivial TF-morphism and is therefore unstable. Thus Lemma \ref{preserving} establishes a method for constructing large unstable circulant graph from smaller ones.
\end{remark}
Jiang and Zhang demonstrated the universality of \cite[Construction 3.1]{JZ2025} in \cite[Proposition 3.2]{JZ2025}. The following proposition specializes  \cite[Proposition 3.2]{JZ2025} to the case of circulant graphs.
\begin{pro}
\label{preserving2}
Let $\Gamma=\Cay(\mathbb{Z}_{2m}, S)$ be a circulant graph admitting a nontrivial TF-morphism $(\alpha,\beta)$ such that both $\alpha$ and $\beta$ preserve $2\mathbb{Z}_{2m}$  set-wise. Then $\Gamma$ satisfies the conditions in Proposition \ref{preserving}.
\end{pro}
Proposition \ref{preserving} serves as a generalization of Lemma \ref{hmmtype3}, while Proposition \ref{preserving2}  establishes that unstable circulant graphs $\Gamma=\Cay(\mathbb{Z}_{2m}, S)$ can be categorized into two distinct types:
\begin{description}
  \item[Type I.] $\Gamma$ satisfies the conditions outlined in Proposition \ref{preserving};
  \item[Type II.] For every nontrivial TF-morphism $(\alpha,\beta)$ of $\Gamma$, either $\alpha$ or $\beta$ fails to preserve $2\mathbb{Z}_{2m}$.
\end{description}

\medskip
We propose the following recursive scheme for classifying unstable circulant graphs of order $2m$:
\begin{description}
  \item[Step 1.] Classify unstable circulant graphs of Type II;
  \item[Step 2.] Classify unstable circulant graphs of order $m$;
  \item[Step 3.] Use the graphs obtained in Step 2 to construct all graphs satisfying Proposition \ref{preserving}.
\end{description}

As an immediate consequence of Remark \ref{rem}, we obtain the following proposition.
\begin{pro}
\label{type2}
Let $\Gamma=\Cay(\mathbb{Z}_{2m}, S)$ be an unstable circulant graph. If the subgraph $\Cay(2\mathbb{Z}_{2m}, S\cap2\mathbb{Z}_{2m})$ of $\Gamma$ is stable, then $\Gamma$ is of Type II.
\end{pro}

\medskip
The classification of all nontrivially unstable circulant graphs $\Gamma=\Cay(\mathbb{Z}_{2m}, S)$ hinges on identifying those of Type II. As demonstrated in \cite{JZ2025}, $\Gamma$ falls under Type I if it satisfies conditions (i), (ii), or (iii) from Lemma \ref{wtype}.  Additionally, if $m$ is even and condition (iv) of Lemma \ref{wtype} holds, $\Gamma$ is also of Type I.  It was proved in \cite{HMM2023} that when $m$ is an odd prime power, $\Gamma$ is unstable if and only if $\Gamma$ meets condition (i) of Lemma \ref{wtype} or $\Gamma\cong\Cay(\mathbb{Z}_{2m}, S+m)$.
Based on these findings, we propose the following conjecture:
\begin{conj}
Let $m$ be an odd integer. If $\Gamma=\Cay(\mathbb{Z}_{2m}, S)$ is a nontrivially unstable circulant graph of Type II, then $\Gamma\cong\Cay(\mathbb{Z}_{2m}, S+m)$.
\end{conj}
If a nontrivially unstable circulant graph $\Gamma=\Cay(\mathbb{Z}_{2m}, S)$ satisfies condition (i) of Theorem \ref{hmmtype}, then $\Gamma$ admits a TF-morphism $(\alpha,\beta)$ defined as follows:
\begin{equation*}
  x^\alpha=\left\{
             \begin{array}{ll}
               x+h, & \hbox{if}~x\in K_o; \\
               x, & \hbox{otherwise,}
             \end{array}
           \right.
~~~\mbox{and}~~~
 x^\beta=\left\{
             \begin{array}{ll}
               x+h, & \hbox{if}~x\in 2K; \\
                x, & \hbox{otherwise,}
             \end{array}
           \right.
\end{equation*}
where $h$ is a nonidentity element of $H$. Obviously, both $\alpha$ and $\beta$ preserve $2\mathbb{Z}_{2m}$ set-wise. By Lemma \ref{preserving2}, $\Gamma$ is of Type I.

If $\Gamma=\Cay(\mathbb{Z}_{4h}, S)$ is a nontrivially unstable circulant graph satisfying the condition in Lemma \ref{oldtonew}, then $\Gamma$ admits a
TF-morphism $(\alpha,\beta)$ defined as follows:
\begin{equation*}
  x^\alpha=\left\{
             \begin{array}{ll}
               x+h, & \hbox{if}~x\in K_o; \\
               x-h, & \hbox{if}~x\in h+K_o; \\
               x, & \hbox{otherwise,}
             \end{array}
           \right.
~~~\mbox{and}~~~
 x^\beta=\left\{
             \begin{array}{ll}
               x+h, & \hbox{if}~x\in 2K; \\
               x-h, & \hbox{if}~x\in h+2K; \\
               x, & \hbox{otherwise.}
             \end{array}
           \right.
\end{equation*}
Furthermore, if $h$ is even, then both $\alpha$ and $\beta$ preserve $2\mathbb{Z}_{4h}$ set-wise. By Lemma \ref{preserving2}, $\Gamma$ is of Type I. If $h$ is odd, the $\Gamma$ may be of  Type II, as illustrated in the following example.

\begin{exam}
\label{20II}
Let $\Gamma=\Cay(\mathbb{Z}_{20},\{\pm1,\pm4,\pm9,10\})$ and $\Gamma'_{e}=\Cay(2\mathbb{Z}_{20},\{\pm4,10\})$. Then $\Gamma$ satisfies the condition in Lemma \ref{oldtonew} and is therefore unstable. Based on the classification of unstable circulant graph of twice an odd prime
\cite[Theorem 5.1.]{HMM2021}, we conclude that $\Gamma'_{e}$ is stable. It follows from Remark \ref{rem} that $\Gamma$ is  of  Type II.
\end{exam}

We end the paper with the following two conjectures.
\begin{conj}
Let $h$ be an odd integer. If $\Gamma=\Cay(\mathbb{Z}_{4h}, S)$ is a nontrivially unstable circulant graph of Type II, then $\Cay\big(\mathbb{Z}_{4h},(S\setminus\{2h\})+2h\big)$ has an automorphism fixing $0$ but moving $2h$.
\end{conj}
\begin{conj}
Let $n$ be a positive integer divisible by $8$. Then every nontrivially unstable circulant graph $\Gamma=\Cay(\mathbb{Z}_{n}, S)$ is of Type I.
\end{conj}

\bigskip
\noindent {\textbf{Acknowledgements}}
\medskip

The author would like to thank Binzhou Xia for his valuable  suggestions, which  improved the clarity and presentation of this paper. This work was supported by the Natural Science Foundation of Chongqing (CSTB2022NSCQMSX1054).
{\small
\end{document}